\documentclass[12pt,fleqn,leqno]{amsart}
\usepackage{amsfonts,amssymb}
\usepackage{mathrsfs}
\usepackage{enumitem}
\usepackage{mathscinet}

\usepackage[svgnames]{xcolor}
\usepackage[colorlinks,linkcolor=blue,citecolor=Green]{hyperref}
\usepackage{titlesec}

\titleformat{\section}[block]
 {\bfseries}
 {\thesection.}
 {\fontdimen2\font}
 {}

\newcommand{\periodafter}[1]{#1.}

\titleformat{\subsection}[runin]
 {\bfseries}
 {\thesubsection.}
 {\fontdimen2\font}
 {\periodafter}

\setlist{noitemsep}
\setenumerate{labelindent=\parindent,label=\upshape{(\alph*)}}

\newtheorem{theorem}{Theorem}[section]
\newtheorem{corollary}[theorem]{Corollary}
\newtheorem{proposition}[theorem]{Proposition}

\theoremstyle{definition}

\newtheorem{example}[theorem]{Example}
\newtheorem{definition}[theorem]{Definition}
\newtheorem{question}{Question}

\DeclareMathOperator{\s}{\mathbb{S}}
\DeclareMathOperator{\uhr}{\upharpoonright}
\DeclareMathOperator{\Z}{\mathbb{Z}}
\DeclareMathOperator{\od}{od}
\DeclareMathOperator{\id}{id}
\DeclareMathOperator{\iR}{\mathbb{P}}

\DeclareMathAlphabet{\mathpzc}{OT1}{pzc}{m}{it}

\DeclareMathOperator{\sel}{\mathpzc{V\mkern-5mu_{cs}}}

\renewcommand{\emptyset}{\varnothing}
\numberwithin{equation}{section}

\begin{document}

\author{Valentin Gutev}
\address{Institute of Mathematics and Informatics, Bulgarian Academy of Sciences,
Acad. G. Bonchev 8, 1113 Sofia, Bulgaria}

\email{\href{mailto:gutev@math.bas.bg}{gutev@math.bas.bg}}



\title{Extension of Hyperspace Selections}

\begin{abstract}
  In 1951, Ernest Michael wrote a definitive seminal article on
  hyperspaces raising a general question that became known as the
  \emph{hyperspace selection problem}. The present paper contains some
  aspects of this problem, along with several open questions. Most of
  these considerations and questions are related to the usual map
  extension problem, but now in the setting of hyperspaces and
  hyperspace selections.
\end{abstract}

\date{\today}
\maketitle

\section{Introduction}

All spaces in this paper are infinite Hausdorff topological spaces. We
will use $\mathscr{F}(X)$ for the set of all nonempty closed subsets
of a space $X$. Also, for a family $\mathscr{V}$ of subsets of $X$,
let
\[
\langle\mathscr{V}\rangle = \left\{S\in \mathscr{F}(X) : S\subset
  \bigcup \mathscr{V}\ \ \text{and}\ \ S\cap V\neq \emptyset,\
  \hbox{for every}\ V\in \mathscr{V}\right\}.
\]
In case $\mathscr{V}=\{V_1,\dots, V_n\}$, we will simply write
$\langle V_1,\dots, V_n\rangle$ instead of
$\langle\mathscr{V}\rangle$.\medskip

The \emph{Vietoris topology} $\tau_V$ on $\mathscr{F}(X)$ is generated
by all collections of the form $ \langle\mathscr{V}\rangle $, where
$\mathscr{V}$ runs over the finite families of open subsets of $X$,
and we refer to $(\mathscr{F}(X),\tau_V)$ as the \emph{Vietoris
  hyperspace} of $X$. In the sequel, any subset
${\mathscr{D}\subset \mathscr{F}(X)}$ will carry the relative Vietoris
topology as a subspace of $(\mathscr{F}(X),\tau_V)$. The following
subspaces of $\mathscr{F}(X)$ will play an important role in this
paper.
\begin{equation}
  \label{eq:ext-sel-v15:1}
  \left\{
    \begin{aligned}
      {[X]^n}&=\{S\subset X: |S| =n\}\ \ \text{and}\ \
      \mathscr{F}_n(X)=\bigcup_{k=1}^n [X]^k,\quad n\geq1,\\[-5pt]
      \Sigma(X)&=\bigcup_{n\geq1} [X]^n=\bigcup_{n\geq1}
      \mathscr{F}_n(X),\ \quad\text{and}\\
      \mathscr{C}(X)&=\{S\in \mathscr{F}(X): S\ \text{is compact}\}.
\end{aligned}\right.
\end{equation}
The singletons of a space $X$ give a natural homeomorphism with the
hyperspace $[X]^1=\mathscr{F}_{1}(X)$, which represents the fact that
the Vietoris topology is \emph{admissible} in the sense of Michael
\cite{michael:51}. \medskip

For a space $X$, a map $f:\mathscr{D}\to X$ is a \emph{selection} for
a subset $\mathscr{D}\subset \mathscr{F}(X)$ if $f(S)\in S$ for
every $S\in\mathscr{D}$, and $f$ is \emph{continuous} if it is
continuous with respect to the Vietoris topology on $\mathscr{D}$. In
what follows, we will use $\sel[\mathscr{D}]$ for the set of all
\emph{Vietoris continuous selections} for $\mathscr{D}$.\medskip

In the present paper, we are interested in the problem of extending
hyperspace selections. If
$\emptyset\neq \mathscr{G}\subset \mathscr{H}\subset \mathscr{F}(X)$,
then $\sel[\mathscr{H}]\neq \emptyset$ implies
$\sel[\mathscr{G}]\neq \emptyset$ because
${h\uhr \mathscr{G}\in \sel[\mathscr{G}]}$ for every
$h\in\sel[\mathscr{H}]$. The inverse implication is not true, which
makes provision for the following two hyperspace selection-extension
properties.

\begin{definition}
  \label{definition-ext-sel-v4:1}
  Let
  $\emptyset\neq \mathscr{G}\subset \mathscr{H}\subset
  \mathscr{F}(X)$. We shall say that the pair
  $(\mathscr{G},\mathscr{H})$ has the \emph{weak selection-extension
    property} if $\sel[\mathscr{G}]\neq \emptyset$ implies
  $\sel[\mathscr{H}]\neq \emptyset$. If moreover
  $\sel[\mathscr{G}]\neq \emptyset$ implies that
  $\sel[\mathscr{G}]=\left\{h\uhr \mathscr{G}: h\in
    \sel[\mathscr{H}]\right\}$, then we will say that
  $(\mathscr{G},\mathscr{H})$ has the \emph{selection-extension
    property}.
\end{definition}

In Definition \ref{definition-ext-sel-v4:1}, the term ``extension'' is
used in the usual context of the map extension problem. Namely, a
selection $h\in\sel[\mathscr{H}]$ is an \emph{extension} of
$g\in \sel[\mathscr{G}]$ if $h\uhr \mathscr{G}=g$. Thus, the pair
$(\mathscr{G},\mathscr{H})$ has the weak selection-extension property
precisely when $\mathscr{G}$ has a selection $g\in \sel[\mathscr{G}]$
which can be extended to a selection $h\in
\sel[\mathscr{H}]$. Similarly, this pair has the selection-extension
property when each selection $g\in\sel[\mathscr{G}]$ can be extended
to a selection $h\in \sel[\mathscr{H}]$. The following trivial
solution of these selection-extension problems is known in different
terms, and will be frequently used in this paper.

\begin{proposition}
  \label{proposition-ext-sel-v32:2}
  For a space $X$ and families
  $\emptyset\neq \mathscr{G}\subset \mathscr{H}\subset \mathscr{F}(X)$
  with $\sel[\mathscr{H}]\neq \emptyset$, the pair
  $(\mathscr{G},\mathscr{H})$ has the weak selection-extension
  property. If moreover $\mathscr{G}$ is $\tau_V$-clopen in
  $\mathscr{H}$, then $(\mathscr{G},\mathscr{H})$ also has the
  selection-extension property.
\end{proposition}

Another aspect of extending hyperspace selections is their natural
relationship with pairs $(Z,X)$ of a space $X$ and a (closed) subspace
$Z\subset X$. Namely, for a space $Y$, let
$\mathscr{D}(Y)\subset \mathscr{F}(Y)$ be a subspace defined by some
property ``$\mathscr{D}$'' of closed subsets, see e.g.\
\eqref{eq:ext-sel-v15:1}. Then for a (closed) subspace $Z\subset X$ of
a space $X$, we have the corresponding pair
$(\mathscr{D}(Z),\mathscr{D}(X))$ of subsets of
$\mathscr{F}(X)$. Accordingly, the (weak) selection-extension property
for this pair can be naturally attributed to the pair $(Z,X)$. In
other words, this represents a special \emph{$\mathscr{D}$-hyperspace}
selection-extension problem for the pair $(Z,X)$.\medskip

The paper is organised as follows. A brief review of orderable-like
spaces and continuous hyperspace selections is given in the next
section. Section \ref{sec:extens-hypersp-over} deals with
$\mathscr{D}$-hyperspace selection-extension properties for pairs
$(Z,X)$ of a space $X$ and a discrete subspace $Z\subset X$. The
essential part of this section is when $\mathscr{D}=\mathscr{F}_2$,
i.e.\ for hyperspaces of at most two-point sets, see
\eqref{eq:ext-sel-v15:1}. Section \ref{sec:extens-hypersp-over-1} is
dedicated to a similar hyperspace selection-extension problem, but now
for pairs $(Z,X)$, where $Z\subset X$ is a closed subspace of $X$. In
Section \ref{sec:extens-probl-hypersp}, we consider the
selection-extension problem for pairs $(\mathscr{G},\mathscr{H})$ of
hyperspaces
$\emptyset\neq\mathscr{G}\subset \mathscr{H}\subset \mathscr{F}(X)$,
where $\mathscr{G}$ is $\tau_V$-dense in $\mathscr{H}$. Section
\ref{sec:weak-extens-probl} deals with a natural reduction of the weak
selection-extension problem for hyperspaces of finite sets to
selections on complete graphs. The last Section
\ref{sec:append-tourn-select} is an appendix providing various
picturesque interpretations of ``weak'' selections on complete graphs
in terms of tournaments on such graphs.

\section{Selections and Orderable-Like Spaces}

A space $X$ is \emph{orderable} (or \emph{linearly ordered}) if it is
endowed with the open interval topology generated by some linear order
on $X$, called \emph{compatible} for $X$. Subspaces of orderable
spaces are not necessarily orderable, they are termed
\emph{suborderable}. A space $X$ is \emph{weakly orderable} if there
exists a coarser orderable topology on $X$ with respect to some linear
order on it (called \emph{compatible} for $X$). The weakly orderable
spaces were introduced by Eilenberg \cite{eilenberg:41}, and are often
called ``Eilenberg orderable''. They were called ``weakly orderable''
in
\cite{zbMATH03355968,zbMATH03379800,dalen-wattel:73,mill-wattel:81}.\medskip

A selection $\sigma:\mathscr{F}_2(X)\to X$ is called a \emph{weak
  selection} for $X$. It generates a natural order-like relation
$\leq_\sigma$ on $X$ \cite[Definition 7.1]{michael:51} defined for
$x,y\in X$ by $x\leq_\sigma y$ iff $\sigma(\{x,y\})=x$. The relation
$\leq_\sigma$ is both \emph{total} and \emph{antisymmetric}, but not
necessarily \emph{transitive}. For convenience, we write $x<_\sigma y$
provided $x\leq_\sigma y$ and $x\neq y$. \label{page-strict-order}
Also, for subsets $A,B\subset X$, we write $A\leq_\sigma B$ if
$x\leq_\sigma y$ for every $x\in A$ and $y\in B$. Similarly,
$A<_\sigma B$ means that $x<_\sigma y$ for every $x\in A$ and
$y\in B$. Finally, we will also use the standard notation for the
intervals generated by $\leq_\sigma$. For instance,
$(\leftarrow, x)_{\leq_\sigma}$ will stand for all $y\in X$ with
$y<_\sigma x$; $(\leftarrow, x]_{\leq_\sigma}$ for that of all
$y\in X$ with $y\leq_\sigma x$; the $\leq_\sigma$-intervals
$(x,\to)_{\leq_\sigma}$, $[x,\to)_{\leq_\sigma}$,
$(x,y)_{\leq_\sigma}$, $[x,y]_{\leq_\sigma}$, etc., are likewise
defined. However, working with such intervals should be done with
caution keeping in mind that the relation $\leq_\sigma$ may lack
transitivity. \medskip

The strict relation $x<_\sigma y$ plays an important role in
describing continuity of a weak selection $\sigma$ for $X$. Namely,
$\sigma$ is continuous iff for every $x,y\in X$ with $x<_\sigma y$,
there are open sets $U,V\subset X$ such that $x\in U$, $y\in V$ and
$U<_\sigma V$, see \cite[Theorem 3.1]{gutev-nogura:01a}. Accordingly,
continuity of weak selections is expressed only in terms of the
elements of $[X]^2$. Moreover, each selection for $[X]^2$ has a unique
extension to a selection for $\mathscr{F}_2(X)$. Based on this, we
will make no difference between selections for $[X]^2$ and those for
$\mathscr{F}_2(X)$. Let us also remark that continuity of a weak
selection $\sigma$ for $X$ implies that all open
$\leq_\sigma$-intervals $(\leftarrow, x)_{\leq_\sigma}$ and
$(x,\to)_{\leq_\sigma}$, $x\in X$, are open in $X$ \cite{michael:51},
but the converse is not true \cite[Example 3.6]{gutev-nogura:01a} (see
also \cite[Corollary 4.2 and Example 4.3]{gutev-nogura:09a}).\medskip

Each weak selection $\sigma$ for $X$ generates a complementary weak
selection $\sigma^*$ defined by $S=\{\sigma(S),\sigma^*(S)\}$ for
$S\in \mathscr{F}_2(X)$. In other words, for $x,y\in X$,
\begin{equation}
  \label{eq:ext-sel-v17:1}
  x\leq_{\sigma^*} y$ \quad \text{if and only if}\quad $y\leq_\sigma x. 
\end{equation}
Accordingly, $\sigma$ is continuous precisely when so is
$\sigma^*$. Moreover, it is evident from \eqref{eq:ext-sel-v17:1} that
${\sigma=\left[\sigma^*\right]^*}$. Motivated by this, we shall say
that weak selections $\sigma$ and $\eta$ for $X$ are
\emph{equivalent}, written $\sigma\sim \eta$, if $\eta=\sigma$ or
$\eta={\sigma}^*$. Here is a simple description of this equivalence
relation. In this property, $\sigma\uhr Z$ is used to denote the
restriction of a weak selection $\sigma$ on $\mathscr{F}_2(Z)$.

\begin{theorem}
  \label{theorem-ext-sel-v45:1}
  If $\sigma$ and $\eta$ are weak selections for a set $X$, then
  $\sigma\sim \eta$ if and only if $\sigma\uhr Z\sim \eta\uhr Z$ for
  every $Z\in [X]^3$. 
\end{theorem}

\begin{proof}
  Evidently, $\sigma\sim \eta$ implies that
  $\sigma\uhr Z\sim \eta\uhr Z$ for every $Z\subset X$. Conversely, if
  $\sigma\not\sim \eta$, then $\sigma\neq \eta\neq \sigma^*$ and there
  are point $x,y,p,q\in X$ such that
  \begin{equation}
    \label{eq:ext-sel-v17:2}
    x<_\sigma y<_\eta x,\quad p<_\sigma q\quad\text{and}\quad p<_\eta
    q. 
  \end{equation}
  Set $\mathscr{Z}(p,q)=\big\{Z\in [X]^3: p,q\in Z\big\}$ and assume
  that $\sigma\uhr Z\sim \eta\uhr Z$ for every
  $Z\in \mathscr{Z}(p,q)$. Then it follows that 
  \begin{equation}
    \label{eq:equivalent-orders}
    \eta\uhr Z= \sigma\uhr Z,\ Z\in
    \mathscr{Z}(p,q),\quad\text{and}\quad S=\{x,y,p,q\}\notin
    \mathscr{Z}(p,q). 
  \end{equation}
  Indeed, by \eqref{eq:ext-sel-v17:1} and \eqref{eq:ext-sel-v17:2},
  $\eta\uhr\{x,y\}={\sigma}^*\uhr\{x,y\}$ and
  $\eta\uhr\{p,q\}=\sigma\uhr\{p,q\}$. Accordingly, the set
  $S=\{x,y,p,q\}$ cannot be a triple, otherwise
  $S\in \mathscr{Z}(p,q)$ will imply that
  $\eta\uhr\{x,y\}=\sigma\uhr\{x,y\}$ which is impossible. \smallskip

  We are also ready to show that, for instance,
  $\eta\uhr Z\not\sim \sigma\uhr Z$, where $Z=\{x,y,p\}$. Namely, if
  $y<_\sigma p$, then $y<_\sigma p<_\sigma q$ and by
  \eqref{eq:equivalent-orders}, $y<_\eta p<_\eta q$. Consequently,
  $x<_\sigma y<_\sigma p$ and $y<_\eta \{x,p\}$. Similarly, if
  $p<_\sigma y$, then $p<_\sigma\{y,q\}$ and for the same reason,
  $p<_\eta \{y,q\}$. Hence, we now have that $p<_\eta y<_\eta x$ and
  $\{p,x\}<_\sigma y$. Thus, $\eta\uhr Z\not\sim \sigma\uhr Z$
  and the proof is complete.
\end{proof}

If a space $X$ is weakly orderable with respect to a linear order
$\leq$, then it also has a continuous weak selection $\sigma$ with
$\leq_\sigma=\leq$. Accordingly, spaces with continuous weak
selections are a natural generalisation of weakly orderable
spaces. Furthermore, in the setting of connected spaces, these are
precisely the weakly orderable spaces. The following result was
obtained in \cite[Lemma 7.2]{michael:51}, its second part is due to
Eilenberg \cite{eilenberg:41}.

\begin{theorem}[\cite{eilenberg:41,michael:51}]
  \label{theorem-ext-sel-v15:2}
  Let $X$ be a connected space with a continuous weak selection
  $\sigma$. Then $\leq_\sigma$ is a linear order on $X$, and $X$ is
  weakly orderable with respect to $\leq_\sigma$. Moreover,
  $\eta\sim \sigma$ for any other continuous weak selection $\eta$ for
  $X$.
\end{theorem}

The question if connectedness in Theorem \ref{theorem-ext-sel-v15:2}
is essential in the presence of compactness was raised by van Douwen
in \cite{douwen:90}. It was resolved by van Mill and Wattel who
obtained the following fundamental result in \cite[Theorem
1.1]{mill-wattel:81}.

\begin{theorem}[\cite{mill-wattel:81}]
  \label{theorem-ext-sel-v47:1}
  A compact space $X$ is orderable if and only if it has a continuous
  weak selection.
\end{theorem}  

In the same paper, the authors posed the following general question
regarding the role of compactness in Theorem
\ref{theorem-ext-sel-v47:1}. It became known as \emph{van Mill and
  Wattel's problem}, or the \emph{weak orderability problem}.

\begin{question}[\cite{mill-wattel:81}]
  \label{question-ext-sel-v19:3}
  If $X$ is a space with a continuous weak selection, then is it true
  that $X$ is weakly orderable?
\end{question}

In 2009, Michael Hru{\v s}{\'a}k and Iv{\'a}n Mart{\'\i}nez-Ruiz
answered Question \ref{question-ext-sel-v19:3} in the negative by
constructing a separable, first countable and locally compact space
which has a continuous weak selection but is not weakly orderable
\cite[Theorem 2.7]{hrusak-martinez:09}. However, the space in their
counterexample is a special Isbell-Mr\'owka space which is not
normal. The question whether the weak orderability problem may depend
on separation axioms was raised in \cite{gutev-nogura:09a}, and is
still open. The interested reader is referred to
\cite{gutev-2013springer,gutev-nogura:09a} for a detailed discussion
on this and other aspects of the hyperspace selection problem.\medskip

The \emph{components} (called also \emph{connected components}) are
the maximal connected subsets of a space $X$. They form a partition of
$X$, and each $p\in X$ is contained in a unique component
$\mathscr{C}[p]$ called the \emph{component} of this point. The
\emph{quasi-component} $\mathscr{Q}[p]$ of $p\in X$ is the
intersection of all clopen subsets of $X$ containing the point
$p$. Evidently, $\mathscr{C}[p]\subset \mathscr{Q}[p]$ but the
converse is not necessarily true.  However, it was shown by R. Duda in
\cite[Theorem 1]{MR0235524} that these components coincide in the
realm of weakly orderable spaces. The result was further extended in
\cite[Theorem 4.1]{gutev-nogura:00b} to all spaces with continuous
weak selections, see also \cite[Corollary 2.3]{MR3430989}.

\begin{theorem}[\cite{gutev-nogura:00b}]
  \label{theorem-ext-sel-v16:1}
  If $X$ is a space with a continuous weak selection, then the
  components of $X$ coincide with the quasi-components.
\end{theorem} 

Regarding the selection problem for weakly orderable spaces, the
following general extension construction is due to Michael \cite[Lemma
7.5.1]{michael:51}.

\begin{theorem}[\cite{michael:51}]
  \label{theorem-hsp-ver-5:7}
  Let $X$ be a weakly orderable space, $\leq$ be a compatible linear
  order on $X$ and
  \begin{equation}
    \label{eq:hsp-ver-13:1}
    \mathscr{K}(X,\leq)=\big\{S\in \mathscr{F}(X): \text{each $T\in
      \mathscr{F}(S)$ has a first $\leq$-element}\big\}.
  \end{equation}
  Then $\mathscr{C}(X)\subset \mathscr{K}(X,\leq)$ and
  ${f(S)=\min_{\leq} S}$, for $S\in \mathscr{K}(X,\leq)$, is a
  continuous selection $f:\mathscr{K}(X,\leq)\to X$ with
  $\leq_f=\leq$.
\end{theorem}

In the setting of connected spaces, Michael used the above
construction to obtain the following crucial result in
\cite[Proposition 7.6]{michael:51}.

\begin{theorem}[\cite{michael:51}]
  \label{theorem-hsp-ver-5:8}
  Let $g$ be a continuous weak selection for a connected space $X$,
  and $\mathscr{K}(X,{\leq_g})$ be as in \eqref{eq:hsp-ver-13:1}. Then
  $g$ can be extended to a continuous selection for
  $\mathscr{K}(X,{\leq_g})$. Moreover, this extension is unique.
\end{theorem}

Using Theorems \ref{theorem-ext-sel-v15:2} and
\ref{theorem-hsp-ver-5:8}, Michael also obtained the following result
in \cite[Proposition 7.8]{michael:51}.

\begin{corollary}
  \label{corollary-hsp-ver-14:1}
  Let $X$ be a connected space with a continuous weak selection, and
  $\mathscr{D}\subset \mathscr{F}(X)$ be such that
  $\mathscr{F}_2(X)\subset \mathscr{D}\subset \mathscr{C}(X)$. Then
  \begin{equation}
    \label{eq:hsp-ver-14:1}
  \big|\sel[\mathscr{F}(X)]\big|\leq
  \big|\sel[\mathscr{D}]\big|=\big|\sel[\mathscr{F}_2(X)]\big|=2.
\end{equation}
\end{corollary}

In the same setting, the case of exactly two continuous selection for
$\mathscr{F}(X)$ was characterised in \cite[Theorem
1]{nogura-shakhmatov:97a}.

\begin{theorem}[\cite{nogura-shakhmatov:97a}]
  \label{theorem-hsp-ver-17:1}
  A connected space $X$ is compact and orderable if and only if it has
  exactly two continuous selections for $\mathscr{F}(X)$.
\end{theorem}

Regarding Theorem \ref{theorem-hsp-ver-17:1}, let us remark that the
disjoint union $X=[0, 1) \oplus [0, 1)$ of two copies of $[0, 1)$ has
exactly two continuous selections for $\mathscr{F}(X)$, but is not
compact \cite[Example 7]{nogura-shakhmatov:97a}.  However, a space $X$
which has exactly one continuous selection for $\mathscr{F}(X)$ must
be connected, see \cite[Lemma 14]{nogura-shakhmatov:97a}. Related to
Theorem \ref{theorem-hsp-ver-17:1}, let us also point out the
following result of Haar and K\"onig \cite{haar-konig:10}.   

\begin{theorem}[\cite{haar-konig:10}]
  \label{theorem-ext-sel-v33:1}
  Let $X$ be a space which is orderable with respect to some linear
  order $\leq$ on it. Then $X$ is compact if and only if each
  nonempty closed subset of $X$ has both a first $\leq$-element and a
  last $\leq$-one.
\end{theorem}

Complementary to Theorem \ref{theorem-ext-sel-v33:1} is the following
property of weakly orderable spaces which is implicitly present in
\cite[Theorem 5.1]{gutev-nogura:01a}.  

\begin{proposition}
  \label{proposition-ext-sel-v34:1}
  Let $X$ be a weakly orderable space with respect to some linear
  order $\leq$ such that each nonempty closed subset of $ X$ has both
  a first $\leq$-element and a last $\leq$-one. Then $X$ is orderable
  with respect to $\leq$, hence compact as well.
\end{proposition}

\begin{proof}
  Take an open set $U\subset X$ with $U\neq X$, and a point $p\in
  U$. Since $S=X\setminus U$ is a nonempty closed set, by condition,
  there are points $s,t\in S$ with $s=\min_\leq S$ and
  $t=\max_\leq S$. If $p=\max_\leq X$, then
  $p\in (s,\to)_{\leq} \subset U$ because $p\in U$. Similarly,
  $p\in (\gets,t)_\leq\subset U$ provided that $p=\min_\leq X$. If
  $x_0=\min_\leq X<p<\max_\leq X=x_1$, set $V=(x_0,x_1)_\leq\cap U$
  and $T=X\setminus V$. Next, consider the nonempty closed sets
  $Y=T\cap (\gets,p]_\leq$ and $Z=T\cap [p,\to)_\leq$. It now follows
  that $p\in (y,z)_\leq\subset V\subset U$, where $y=\max_\leq Y$ and
  $z=\min_\leq Z$. Thus, $X$ is orderable with respect to $\leq$ and
  by Theorem \ref{theorem-ext-sel-v33:1}, it is also compact.
\end{proof}

Regarding the selection problem for compact-like spaces, several
authors contributed to the following fundamental result.

\begin{theorem}[\cite{artico-marconi-pelant-rotter-tkachenko:02, douwen:90, 
  garcia-ferreira-sanchis:04, glicksber:59, miyazaki:01b,
  venkataraman-rajagopalan-soundararajan:72}]
  \label{theorem-ext-sel-v21:1}
  Every countably compact space with a continuous weak selection is
  sequentially compact, and every pseudocompact space with a
  continuous weak selection is suborderable. In particular, for a
  Tychonoff space $X$ with a continuous weak selection, the following
  are equivalent\textup{:}
  \begin{enumerate}
  \item $X$ is countably compact.
  \item $X$ is pseudocompact.
  \item $X$ is sequentially compact. 
  \end{enumerate}
\end{theorem}

In fact, the key element in the proof of Theorem
\ref{theorem-ext-sel-v21:1} was the following selection-extension
result to the \emph{{\v C}ech-Stone compactification} $\beta X$ of a
pseudocompact space $X$, it was explicitly summarised in \cite[Theorem
3.7 and Corollary 3.12]{gutev-2013springer}.

\begin{theorem}
  \label{theorem-ext-sel-v13:1}
  If $X$ is a pseudocompact space, then each continuous weak selection
  for $X$ can be extended to a continuous weak selection for the {\v
    C}ech-Stone compactification of $X$. 
\end{theorem}

Regarding the role of the {\v C}ech-Stone compactification in such
hyperspace selec\-tion-extension properties, let us also mention the
following result.

\begin{theorem}[\cite{mill-wattel:84}]
  \label{theorem-compact-3}
  A Tychonoff space $X$ is suborderable if and only if it has a
  continuous weak selection $\sigma$ such that for every
  $ p\in \beta X\setminus X$, the selection $\sigma$ can be extended
  to a continuous weak selection for $X\cup\{p\}$.
\end{theorem}

Answering a question of \cite{MR3122363} about the role of countable
compactness in Theorem \ref{theorem-ext-sel-v21:1}, see also
\cite[Problem 3.10]{gutev-2013springer}, the following complementary
result was obtained in \cite[Theorems 1.5 and 1.6]{Motooka2019}.

\begin{theorem}[\cite{Motooka2019}]
  \label{theorem-ext-sel-v37:1}
  Each countably compact space $X$ with a continuous weak selection is
  weakly orderable. If moreover $X$ is regular, then it is also
  suborderable.
\end{theorem}

A space $X$ is \emph{totally disconnected} if each quasi-component is
a singleton or, equivalently, if each point of $X$ is an intersection
of clopen sets. A space $X$ is \emph{zero-dimensional} if it has a
base of clopen sets, and $X$ is \emph{strongly zero-dimensional} if
its \emph{covering dimension} is zero. \medskip

A metric $d$ on a set $X$ is \emph{non-Archimedean}, or an
\emph{ultrametric}, if
\[
d(x,y)\leq \max\big\{d(x,z),d(z,y)\big\}\quad \hbox{for every
  $x,y,z\in X$.}
\]
In this case, we say that $(X,d)$ is a \emph{non-Archimedean metric
  space}, or simply an \emph{ultrametric space}. It was shown by J. de
Groot \cite{groot:56} that a metrizable space $X$ has a compatible
non-Archimedean metric iff it is strongly zero-dimensional. Moreover,
it was shown by Papi\v{c} \cite{MR0073964} and Kurepa \cite{MR0102789}
that a non-Archimedean metrizable space is orderable. Generalising a
result of I.\,L. Lynn \cite{MR0138089}, H. Herrlich showed in
\cite{herrlich:65} that a totally disconnected metrizable space $X$ is
orderable precisely when it is strongly zero-dimensional. These
results are summarised below.

\begin{theorem}
  \label{theorem-hsp-ver-10:9}
  For a metrizable space $X$, the following are
  equivalent\textup{:}
  \begin{enumerate}
  \item $X$ is totally disconnected and orderable.
  \item $X$ is strongly zero-dimensional.
  \item $X$ has a compatible non-Archimedean metric. 
  \end{enumerate}
\end{theorem}

The following theorem is complementary to Theorem
\ref{theorem-hsp-ver-10:9}. It was obtained by Choban \cite[Corollary
2.1]{choban:70a}, using a different approach it was also obtained by
Engelking, Heath and Michael \cite[Corollary
1.2]{engelking-heath-michael:68}.

\begin{theorem}[\cite{choban:70a,engelking-heath-michael:68}] 
  \label{theorem-ext-sel-v26:4}
  Each completely metrizable strongly zero-dimensional space $X$ has a
  continuous selection for $\mathscr{F}(X)$.
\end{theorem}

Finally, let us explicitly remark that each metrizable space $X$ with
a continuous selection for $\mathscr{F}(X)$ is completely
metrizable \cite[Corollary 3.5]{mill-pelant-pol:96}, see also
\cite{zbMATH06882273}.

\begin{theorem}[\cite{mill-pelant-pol:96}]
  \label{theorem-ext-sel-v47:2}
  If a metrizable space $X$ has a continuous selection for
  $\mathscr{F}(X)$, then it is completely metrizable.
\end{theorem}

\section{Extension of Selections and Discrete Sets} 
\label{sec:extens-hypersp-over}

In this section, we first consider the hyperspace selection-extension
problem for weak selections defined on discrete subspaces $Z\subset X$
of a space $X$. Evidently, each weak selection for such a subspace $Z$
is continuous. In the special case when $Z\subset X$ is a triple,
Theorems \ref{theorem-ext-sel-v45:1} and \ref{theorem-ext-sel-v15:2}
imply the following description of totally disconnected spaces.

\begin{theorem}
  \label{theorem-ext-sel-v4:1}
  Let $X$ be a space with a continuous weak selection. Then $X$ is
  totally disconnected if and only if
  \begin{equation}
    \label{eq:ext-sel-v4:3}
    \sel[\mathscr{F}_2(Z)]=\big\{\sigma\uhr \mathscr{F}_2(Z):
    \sigma\in \sel[\mathscr{F}_2(X)]\big\}\ \ \text{for every $Z\in[X]^3$.}
  \end{equation}
  In particular, $X$ is totally disconnected precisely when for every
  $Z\in [X]^3$ there are selection
  $\sigma, \eta\in\sel[\mathscr{F}_2(X)]$ with
  $\eta\uhr Z\not\sim \sigma\uhr Z$.
\end{theorem}

\begin{proof}
  Suppose that \eqref{eq:ext-sel-v4:3} holds, but $X$ is not totally
  disconnected. According to Theorem \ref{theorem-ext-sel-v16:1}, $X$
  has a connected component $Y$ with $|Y|\geq 2$. Take a continuous
  weak selection $\eta$ for $X$, a triple $Z\in [Y]^3$ and a weak
  selection $\gamma$ for $Z$ with $\gamma\not\sim \eta\uhr Z$. Then by
  \eqref{eq:ext-sel-v4:3}, $\gamma$ can be extended to a continuous
  weak selection $\sigma$ for $X$. Hence, by Theorem
  \ref{theorem-ext-sel-v15:2}, $\eta\uhr Y\sim \sigma\uhr Y$. However,
  by Theorem \ref{theorem-ext-sel-v45:1}, this is impossible because
  $\eta\uhr Z\not\sim \sigma\uhr Z$. Thus, $X$ must be totally
  disconnected.\smallskip

  Conversely, suppose that $X$ is totally disconnected and
  $Z\in [X]^3$. Then exists a clopen partition
  $\left\{U_z:z\in Z\right\}$ of $X$ such that $z\in U_z$ for every
  $z\in Z$. Take a weak selection $\gamma$ for $Z$ and
  $\sigma\in\sel[\mathscr{F}_2(X)]$. Next, for every
  $S\in \mathscr{F}_2(X)$, let
  $Z_S=\left\{z\in Z: S\cap U_z\neq\emptyset \right\}\in
  \mathscr{F}_2(Z)$.  Finally, set
  $\eta(S)=\sigma\left(S\cap U_{\gamma(Z_S)}\right)$ for every
  $S\in \mathscr{F}_2(X)$. Clearly, $\eta:\mathscr{F}_2(X)\to X$ is a
  continuous selection. Moreover, it is an extension of $\gamma$
  because
  $\eta(S)=\sigma\left(S\cap
    U_{\gamma(Z_S)}\right)=\sigma(\{\gamma(S)\})=\gamma(S)$ for every 
  $S\in \mathscr{F}_2(Z)$.
\end{proof}

It is evident that the second part of the proof of Theorem
\ref{theorem-ext-sel-v4:1} does not depend on the number of the
elements of $Z$ as far as this number is finite. Accordingly, we also
have the following consequence from this proof.

\begin{corollary}
  \label{corollary-ext-sel-v5:1}
  Let $X$ be a space which has a continuous weak selection. Then $X$
  is totally disconnected if and only if for each $Z\in\Sigma(X)$,
  each weak selection for $Z$ can be extended to a continuous weak
  selection for $X$.
\end{corollary}

In fact, the hyperspace selection-extension property in Theorem
\ref{theorem-ext-sel-v4:1} remains valid for arbitrary discrete sets
$Z\in \mathscr{F}(X)$ under the following extra condition. 

\begin{proposition}
  \label{proposition-ext-sel-v9:1}
  Let $X$ be a space and $Z\in \mathscr{F}(X)$ be such that there
  exists a discrete collection $\{U_z:z\in Z\}$ of clopen sets with
  $z\in U_z$, for each $z\in Z$. If $X$ has a continuous weak
  selection, then each weak selection for $Z$ can be extended to a
  continuous weak selection for $X$.
\end{proposition}

\begin{proof}
  Briefly, the condition is equivalent to the existence of a clopen
  partition $\{U_z:z\in Z\}$ of $X$ with $z\in U_z$, for each
  $z\in Z$. Hence, we can proceed precisely as in the second part of
  the proof of Theorem \ref{theorem-ext-sel-v4:1}.
\end{proof}

For countable discrete subspaces $Z\in \mathscr{F}(X)$, we also have
the following partial inverse of Proposition
\ref{proposition-ext-sel-v9:1}.

\begin{theorem}
  \label{theorem-ext-sel-v48:1}
  Let $X$ be a space such that for every countable discrete
  $Z\in \mathscr{F}(X)$, each weak selection for $Z$ can be extended
  to a continuous weak selection for $X$.  Then for every countable
  discrete $Z\in \mathscr{F}(X)$ there exists a pairwise-disjoint
  collection $\{U_z:z\in Z\}$ of clopen sets such that $z\in U_z$ for
  each $z\in Z$.
\end{theorem}

\begin{proof}
  We follow the idea in \cite[Proposition 2.2 and Lemma
  2.3]{Gutev2021a}. Namely, take an infinite countable discrete
  $Z\in \mathscr{F}(X)$, a partition $\mathscr{Z}\subset [Z]^3$ of $Z$
  into triples, and an enumeration $\mathscr{Z}=\{Z_n: n<\omega\}$ of
  this partition. Next, take a weak selection $\eta$ for $Z$ with the
  property that $\leq_\eta$ is not transitive on each $Z_n$,
  $n<\omega$, and $Z_k<_\eta Z_n$ precisely when $k<n$. Then, by
  condition, $\eta$ can be extended to a continuous weak selection
  $\sigma$ for $X$. For convenience, set $Z_{-1}=\emptyset$ and for
  each $n<\omega$, define an open subset $W_n\subset X$ by
  \begin{equation}
    \label{eq:ext-sel-v48:1}
    W_n=\left\{x\in X: Z_{k}<_\sigma x<_\sigma Z_{n+1}\ \text{for every
        $k<n$}\right\}. 
  \end{equation}
  Evidently, $Z_n\subset W_n$. Moreover, we also have that
  \begin{equation}
    \label{eq:ext-sel-v48:2}
    W_m\cap W_n=\emptyset\quad\text{for every $m\geq n+2$.}
  \end{equation}
  This is a simple consequence of \eqref{eq:ext-sel-v48:1} because
  $W_m\subset\left\{x\in X: Z_{n+1}<_\sigma x\right\}$ and
  $W_n\subset \left\{x\in X: x<_\sigma Z_{n+1}\right\}$.\smallskip

  We can now slightly shrink any $W_n$ to a clopen subset containing
  $Z_n$, namely for each $n<\omega$ there exists a clopen set
  $H_n\subset X$ such that
  \begin{equation}
    \label{eq:ext-sel-v48:3}
    Z_n\subset H_n\subset W_n.
  \end{equation}
  Indeed, the set $W_n$ is an intersection of the
  $\leq_\sigma$-intervals $(\gets,z)_{\leq_\sigma}$ for
  $z\in Z_{n+1}$, and $(z,\to)_{\leq_\sigma}$ for $z\in Z_k$,
  $0\leq k<n$. Moreover, the sets $Z_{n+1},Z_k$, $k<n$, are finite and
  disjoint from $Z_n$. However, by Theorem \ref{theorem-ext-sel-v4:1},
  $X$ is totally disconnected. Therefore, the triple $Z_n$ is
  contained in a clopen set $T_n\subset X$ such that
  $T_n\cap Z_{n+1}=\emptyset$ and $T_n\cap Z_k=\emptyset$,
  $k<n$. Accordingly, $H_n=W_n\cap T_n$ is as required in
  \eqref{eq:ext-sel-v48:3}. \smallskip

  We can finish the proof as follows. Since
  $\sigma\uhr Z_n=\eta\uhr Z_n$ and $\leq_\eta$ is not transitive on
  $Z_n$, this set is \emph{$\sigma$-circular} in the sense of
  \cite{Gutev2021a}. In other words, $Z_n$ consists of points
  $u,v,w\in Z_n$ such that $u<_\sigma v<_\sigma w<_\sigma
  u$. Accordingly, each two of these points define a clopen
  neighbourhood of the other, namely
  \begin{equation}
    \label{eq:ext-sel-v48:4}
    v\in
    T_v=(u,w)_{\leq_\sigma},\quad u\in T_u=(w,v)_{\leq_\sigma}\quad
    \text{and}\quad w\in T_w=(v,u). 
  \end{equation}
  It is also evident that these $\leq_\sigma$-intervals are
  pairwise-disjoint. We can now take $U_z=T_z\cap H_n$ for $z\in Z_n$
  and $n<\omega$. Accordingly, by \eqref{eq:ext-sel-v48:3}, we get a
  family $\{U_z:z\in Z\}$ of clopen sets with $z\in U_z$, $z\in Z$. To
  show that it is also pairwise-disjoint take points $y,z\in Z$ with
  $y\neq z$. If $y,z\in Z_n$ for some $n<\omega$, then
  $U_y\cap U_z=\emptyset$ because $T_y\cap T_z=\emptyset$, see
  \eqref{eq:ext-sel-v48:4}. Suppose that $y\in Z_m$ and $z\in Z_n$ for
  some $m>n$. Then by \eqref{eq:ext-sel-v48:2} and
  \eqref{eq:ext-sel-v48:3}, $U_y\cap U_z\subset W_m\cap W_n=\emptyset$
  provided that $m\geq n+2$. So, the remaining case is when
  $m=n+1$. In this case, by \eqref{eq:ext-sel-v48:4}, there is a point
  $x\in Z_m=Z_{n+1}$ such that $U_y\subset
  (x,\to)_{\leq_\sigma}$. However, by \eqref{eq:ext-sel-v48:1}, we
  also have that $U_z\subset W_n\subset
  (\gets,x)_{\leq_\sigma}$. Thus, we get again that
  $U_y\cap U_z=\emptyset$, and the proof is complete.
\end{proof}

Proposition \ref{proposition-ext-sel-v9:1} and Theorem
\ref{theorem-ext-sel-v48:1} are naturally related to some separation
properties of topological spaces. A space $X$ is (\emph{strongly})
\emph{$\tau$-collectionwise Hausdorff} if for each closed discrete set
$Z\subset X$ of cardinality at most $\tau$, there exists a
pairwise-disjoint (resp., discrete) collection $\{U_z:z\in Z\}$ of
open sets such that $z\in U_z$ for all $z\in Z$. A space $X$ is
(\emph{strongly}) \emph{collectionwise Hausdorff} if it is (strongly)
$\tau$-collectionwise Hausdorff for each cardinal number $\tau$. In
what follows, we will use the abbreviation ``cwH'' for
``collectionwise Hausdorff''. In these terms, a space $X$ as in
Theorem \ref{theorem-ext-sel-v48:1} is both totally disconnected and
$\omega$-cwH. Accordingly, we have the following natural question.

\begin{question}
  \label{question-ext-sel-v11:1}
  Let $X$ be a space such that for every discrete
  $Z\in \mathscr{F}(X)$, each weak selection for $Z$ can be
  extended to a continuous weak selection for $X$.  Then, is it
  true that for every discrete $Z\in \mathscr{F}(X)$ there exists a
  pairwise-disjoint collection $\{U_z:z\in Z\}$ of clopen sets such
  that $z\in U_z$ for each $z\in Z$? In particular, is it true that
  $X$ is (strongly) cwH?
\end{question}

Proposition \ref{proposition-ext-sel-v9:1} suggests another natural
question.

\begin{question}
  \label{question-ext-sel-v12:1}
  Let $X$ be a totally disconnected strongly cwH space. If $X$ has a
  continuous weak selection, then is it true that for every discrete
  $Z\in \mathscr{F}(X)$, each weak selection for $Z$ can be extended
  to a continuous weak selection for $X$?
\end{question}

For a zero-dimensional space $X$, the answer to Question
\ref{question-ext-sel-v12:1} is ``yes''. Namely, we have the following
immediate consequence of Proposition \ref{proposition-ext-sel-v9:1}.

\begin{corollary}
  \label{corollary-ext-sel-v45:1}
  Let $X$ be a zero-dimensional strongly cwH space. If $X$ has a
  continuous weak selection, then for each discrete
  $Z\in \mathscr{F}(X)$, each weak selection for $Z$ can be extended
  to a continuous weak selection for $X$.
\end{corollary}

Thus, we also have the following complementary question.

\begin{question}
  \label{question-ext-sel-v45:1}
  Let $X$ be a zero-dimensional $\tau$-cwH space with a continuous
  weak selection. Then, is it true that for every discrete
  $Z\in \mathscr{F}(X)$ with $|Z|\leq\tau$, each weak selection for
  $Z$ can be extended to a continuous weak selection for $X$?
\end{question}

A space $X$ is \emph{collectionwise normal} \cite{bing:51} if for
every discrete collection $\mathscr{D}$ of (closed) subsets of $X$
there is a discrete collection $\{U_D:D\in \mathscr{D}\}$ of open
subsets of $X$ such that ${D\subset U_D}$ for all $D\in
\mathscr{D}$. Collectionwise normality implies normality but the
converse is not true. In \cite[Examples G and H]{bing:51}, Bing
described a normal space which is not collectionwise normal, see also
\cite[5.1.23 Bing's Example]{engelking:89}. If a space $X$ is
collectionwise normal with respect to discrete families $\mathscr{D}$
of cardinality at most $\tau$, then it is called
\emph{$\tau$-collectionwise normal}. It is well known that for every
infinite cardinal $\tau$ there exists a $\tau$-collectionwise normal
space which is not $\tau^+$-collectionwise normal
\cite{przymusinski:78a}, where the cardinal $\tau^+$ is the immediate
successor of~$\tau$. In this setting, a space $X$ is normal precisely
when it is $\omega$-collectionwise normal, which follows easily from
Tietze-Urysohn extension theorem.\medskip

Since each normal (i.e.\ $\omega$-collectionwise normal) space is
strongly $\omega$-cwH, we have the following special case of Corollary
\ref{corollary-ext-sel-v45:1}.

\begin{corollary}
  \label{corollary-ext-sel-v45:2}
  Let $X$ be a zero-dimensional normal space. If $X$ has a continuous
  weak selection, then for each countable discrete
  $Z\in \mathscr{F}(X)$, each weak selection for $Z$ can be extended
  to a continuous weak selection for $X$.
\end{corollary}

This brings the following related question.

\begin{question}
  \label{question-ext-sel-v18:2}
  Let $X$ be a totally disconnected normal which has a continuous weak
  selection. Then, is it true that for every countable discrete
  $Z\in \mathscr{F}(X)$, each weak selection for $Z$ can be extended
  to a continuous weak selection for $X$?
\end{question}

A space $X$ is \emph{ultranormal} if every two disjoint closed subsets
are contained in disjoint clopen subsets; equivalently, if its
\emph{large inductive dimension} is zero. In the realm of normal
spaces, $X$ is ultranormal if and only if it is strongly
zero-dimensional. Each totally disconnected suborderable space is
ultranormal. This was essentially shown in H. Herrlich \cite[Lemma
1]{MR0185564} and explicitly stated in S. Purisch \cite[Proposition
2.3]{purisch:77}.\label{page:suborderable-zero} Moreover, each
suborderable space is both collectionwise normal and countably
paracompact, see \cite{MR0063646, MR0093753, MR0257985}.  Thus, we
have the following consequence from Corollary
\ref{corollary-ext-sel-v5:1} and Proposition
\ref{proposition-ext-sel-v9:1}.

\begin{corollary}
  \label{corollary-ext-sel-v18:1}
  For a suborderable space $X$, the following are equivalent\textup{:}
  \begin{enumerate}
  \item\label{item:ext-sel-v12:1} $X$ is strongly zero-dimensional.
  \item\label{item:ext-sel-v12:2} For every discrete
    $Z\in \mathscr{F}(X)$, each weak selection for $Z$ can be extended
    to a continuous weak selection for $X$.
  \item\label{item:ext-sel-v12:3} For every $Z\in \Sigma(X)$, each
    weak selection for $Z$ can be extended to a continuous weak
    selection for $X$.
  \end{enumerate}
\end{corollary}

Regarding the role of collectionwise normality, it is interesting to
consider the following special case of Question
\ref{question-ext-sel-v12:1}.

\begin{question}
  \label{question-ext-sel-v12:2}
  Let $X$ be a totally disconnected collectionwise normal space with a
  continuous weak selection. Then, is it true that for every discrete
  $Z\in \mathscr{F}(X)$, each weak selection for $Z$ can be extended
  to a continuous weak selection for $X$? What about if $X$ is
  paracompact or even metrizable?
\end{question}

The answer to Question \ref{question-ext-sel-v12:2} is ``yes'' if $X$
is also locally compact, which follows from Corollary
\ref{corollary-ext-sel-v45:1} because each totally disconnected
locally compact space is zero-dimensional, see \cite[Theorem
6.2.9]{engelking:89}. Related to this, let us remark that according to
\cite[Corollary 6.14]{gutev:07b}, a locally compact totally
disconnected paracompact space $X$ is orderable if and only if it has
a continuous weak selection. Thus, in this case, $X$ is an ultranormal
paracompact space. These spaces are often called
\emph{ultraparacompact} and have the property that every open cover
has a pairwise disjoint open refinement, see \cite{MR0261565}.\medskip

Regarding the role of the suborderability, we also have the
following question.

\begin{question}
  \label{question-ext-sel-v46:1}
  If $X$ is a (strongly) zero-dimensional weakly orderable space, then
  is it true that for every discrete $Z\in \mathscr{F}(X)$, each weak
  selection for $Z$ can be extended to a continuous weak selection for
  $X$?
\end{question}

Another interesting case of the hyperspace selection-extension problem
for discrete subspaces $Z\subset X$ is when $Z$ is not necessarily
closed in $X$. The simplest possible case is when $Z$ is a countable
discrete subspace.

\begin{proposition}
  \label{proposition-ext-sel-v13:2}
  Let $X$ be a space such that for every countable discrete subspace
  $Z\subset X$, each weak selection for $Z$ can be extended to a
  continuous weak selection for $X$. Then $X$ has no non-trivial
  convergent sequences.
\end{proposition}

\begin{proof}
  Assume that $X$ has a point $p\in X$ which is the limit of a
  sequence of points of $X\setminus \{p\}$. Then there are sequences
  $S=\{y_n\},T=\{z_n\}\subset X\setminus \{p\}$ such that
  $S\cap T=\emptyset$ and
  $\lim_{n\to \infty}y_n=p=\lim_{n\to \infty}z_n$. We can further
  assume that there is a point $q\in X\setminus \{p\}$ with
  $q\notin S\cup T$. Set $Z=S\cup\{q\}\cup T$ and define a weak
  selection $\eta$ for $Z$ such that
  $q<_\eta S<_\eta T<_\eta q$. Then by condition, $\eta$ can
  be extended to a continuous weak selection $\sigma$ for $X$. However,
  $q<_\sigma S$ implies that
  $\sigma(\{q,p\})= \lim_{n\to \infty}\eta(\{q,y_n\})=q$, so
  $q<_\sigma p$. Similarly, $T<_\sigma q$ implies that $p<_\sigma
  q$. Clearly, this is impossible.
\end{proof}

A space $X$ is \emph{scattered} if every nonempty subset $A\subset X$
has an element $p\in A$ which is isolated relative to $A$, i.e. $p$ is
an isolated point of $A$. Equivalently, $X$ is scattered if every
nonempty closed subset of $X$ has an isolated point. 

\begin{corollary}
  \label{corollary-ext-sel-v13:3}
  Let $X$ be a space such that for every countable discrete subspace
  $Z\subset X$, each weak selection for $Z$ can be extended to a
  continuous weak selection for $X$. If $\mathscr{F}(X)$ has a
  continuous selection, then $X$ is a scattered space.
\end{corollary}

\begin{proof}
  Take a selection $f\in \sel[\mathscr{F}(X)]$ and a closed set
  $T\in \mathscr{F}(X)$. Then $f(T)\in T$ is an isolated point of
  $T$. Indeed, if $f(T)$ is not isolated in $T$, then by
  \cite[Proposition
  4.4]{garcia-ferreira-gutev-nogura-sanchis-tomita:99}, $T$ contains a
  non-trivial convergent sequence. However, by Proposition
  \ref{proposition-ext-sel-v13:2}, this is impossible. Thus, $X$ is a
  scattered space.
\end{proof}

Regarding the role of the condition ``$\sel[\mathscr{F}(X)]\neq
\emptyset$'' in Corollary \ref{corollary-ext-sel-v13:3}, we have the
following question.  

\begin{question}
  \label{question-ext-sel-v13:4}
  Let $X$ be a space such that for every countable discrete subspace
  $Z\subset X$, each weak selection for $Z$ can be extended to a
  continuous weak selection for $X$. Then, is it true that $X$ is a
  scattered space?
\end{question}

A selection $f:\mathscr{F}(X)\to X$ is  a \emph{zero-selection}
if $f(S)$ is an isolated point of $S$ for every $S\in
\mathscr{F}(X)$. Evidently, each scattered space has a
zero-selection, and every space which has a zero-selection is
scattered. Continuous zero-selections imply some interesting
properties. For instance, it was shown in \cite{fujii-nogura:99} that
a compact space $X$ is an ordinal space if and only if
$\mathscr{F}(X)$ has a continuous zero-selection. Here, $X$ is an
\emph{ordinal space} if it is an ordinal equipped with the open
interval topology. For some related results in the non-compact case,
the interested reader is referred to \cite{artico-marconi:00,
  artico-marconi-pelant-rotter-tkachenko:02,
  fujii:02,gutev:00e,gutev-nogura:00d}. The proof of Corollary
\ref{corollary-ext-sel-v13:3} implies the following more restrictive
property on the continuous selections for $\mathscr{F}(X)$.

\begin{corollary}
  \label{corollary-ext-sel-v13:4}
  Let $X$ be a space such that for every countable discrete subspace
  $Z\subset X$, each weak selection for $Z$ can be extended to a
  continuous weak selection for $X$. Then each
  $f\in \sel[\mathscr{F}(X)]$ is a zero-selection for
  $\mathscr{F}(X)$.
\end{corollary}

\begin{proof}
  As shown in the proof of Corollary \ref{corollary-ext-sel-v13:3}, for
  $f\in \sel[\mathscr{F}(X)]$ and $T\in \mathscr{F}(X)$, the point $f(T)$
  is isolated in $T$.
\end{proof}

The property obtained in Corollary \ref{corollary-ext-sel-v13:4} is
stronger than that of $X$ being a scattered space. It is evident from
the following further consequence.

\begin{corollary}
  \label{corollary-ext-sel-v13:5}
  Let $X$ be a compact space such that for every countable discrete
  subspace ${Z\subset X}$, each weak selection for $Z$ can be extended
  to a continuous weak selection for $X$. Then $X$ is a finite set.
\end{corollary}

\begin{proof}
  Since $X$ is compact and has a continuous weak selection, it follows
  from \cite[Theorem 1.1]{mill-wattel:81} that
  $\sel[\mathscr{F}(X)]\neq\emptyset$. Hence, by Corollary
  \ref{corollary-ext-sel-v13:4}, $\mathscr{F}(X)$ has a continuous
  zero-selection. Accordingly, by \cite[Theorem 1]{fujii-nogura:99},
  $X$ is an ordinal space. However, by Proposition
  \ref{proposition-ext-sel-v13:2}, $X$ has no non-trivial convergent
  sequences. Therefore, $X$ must be a finite set.  
\end{proof}

A space $X$ is \emph{sequential} if for each set $A\subset X$, each
point of $\overline{A}$ is the limit of a sequence of points of
$A$. Regarding discrete spaces, we also have the following immediate
consequence of Proposition \ref{proposition-ext-sel-v13:2}.

\begin{corollary}
  \label{corollary-ext-sel-v13:2}
  Let $X$ be a sequential space such that for every countable discrete
  subspace $Z\subset X$, each weak selection for $Z$ can be extended
  to a continuous weak selection for $X$. Then $X$ is a discrete
  space.
\end{corollary}

Sequential spaces are very similar to the \emph{countably tight}
spaces, namely to those spaces $X$ in which for every $A\subset X$ and
a point $p\in \overline{A}$, there exists a countable set $Z\subset A$
with $p\in \overline{Z}$.

\begin{question}
  \label{question-ext-sel-v13:3}
  Let $X$ be a countably tight space such that for every countable
  discrete subspace $Z\subset X$, each weak selection for $Z$ can be
  extended to a continuous weak selection for $X$. Then, is it true
  that $X$ is discrete?
\end{question}

For arbitrary discrete subspaces, we also have the following general
question.

\begin{question}
  \label{question-ext-sel-v13:3}
  Let $X$ be a space such that for every discrete subspace
  ${Z\subset X}$, each weak selection for $Z$ can be extended to a
  continuous weak selection for $X$. Then what can one say about $X$?
  Does there exist a non-discrete space $X$ with this property?
\end{question}

In the setting of general continuous selections for hyperspaces of 
discrete subspaces, we have the following alternative version of
Proposition \ref{proposition-ext-sel-v9:1}.

\begin{proposition}\label{pr:extension-1}
  Let $X$ be a space and $Z\in \mathscr{F}(X)$ be such that there
  exists a discrete collection $\{U_z:z\in Z\}$ of clopen sets with
  $z\in U_z$, for each $z\in Z$. If $\mathscr{F}(X)$ has a continuous
  selection, then each continuous selection for $\mathscr{F}(Z)$ can
  be extended to a continuous selection for $\mathscr{F}(X)$.
\end{proposition}

\begin{proof}
  The proof is almost identical to that one in Theorem
  \ref{theorem-ext-sel-v4:1}. Namely, as in the proof of Proposition
  \ref{proposition-ext-sel-v9:1}, the discrete collection
  $\{U_z:z\in Z\}$ can be assumed to be a cover of $X$. Next, let
  $Z_S=\left\{z\in Z: S\cap U_z\neq\emptyset\right\}\in
  \mathscr{F}(Z)$ for $S\in \mathscr{F}(X)$. Then $Z_S=S$ for
  every $S\in \mathscr{F}(Z)$. Hence, for $g\in \sel[\mathscr{F}(Z)]$
  and $f\in\sel[\mathscr{F}(X)]$, we can define an extension
  $h\in \sel[\mathscr{F}(X)]$ of $g$ by
  $h(S)=f\left(S\cap U_{g(Z_S)}\right)$, $S\in \mathscr{F}(X)$.
\end{proof}

Regarding Proposition \ref{pr:extension-1}, let us explicitly remark
that each discrete space $Z$ has a continuous selection for
$\mathscr{F}(Z)$.\label{page-ref-discrete} For instance, take a
well-order $\leq$ on $Z$, and define $g:\mathscr{F}(Z)\to Z$ by
$g(S)=\min S$ for $S\in \mathscr{F}(Z)$. However, not every selection
for $\mathscr{F}(Z)$ is continuous. In contrast, every selection for
$\Sigma(Z)=\mathscr{C}(Z)$ is continuous because the Vietoris
hyperspace $\Sigma(Z)$ is discrete. On the other hand, $\Sigma(Z)$ is
not necessarily $\tau_V$-closed in $\mathscr{F}(X)$. In fact, it is
$\tau_V$-dense in $\mathscr{F}(Z)$ which brings the following
characterisation of finite discrete spaces.

\begin{proposition}
  \label{proposition-ext-sel-v25:1}
  A discrete space $Z$ is finite if and only if each selection for
  $\Sigma(Z)$ can be extended to a continuous selection for
  $\mathscr{F}(Z)$.
\end{proposition}

\begin{proof}
  If $Z$ is finite, then $\mathscr{F}(Z)=\Sigma(Z)$ and there is
  nothing to prove. Conversely, assume that $Z$ is an infinite
  discrete space possessing the above hyperspace selection-extension
  property. Then $Z$ has a partition into sets $S$ and $H$ such that
  $S$ is countable and infinite. Take an enumeration
  $S=\{s_0,s_1,\dots, s_n,\dots\}$ of $S$, and define a selection
  $\sigma:\Sigma(S)\to S$ with the property that
  $\sigma(\{s_0,\dots, s_n\})=s_n$ for every $n<\omega$. Also, take
  any selection $\eta:\Sigma(H)\to H$, and define a selection
  $g:\Sigma(Z)\to Z$ by $g(T)=\sigma(T)$ if $T\in \Sigma(S)$, and
  $g(T)=\eta(T\cap H)$ otherwise.  Then by condition, $g$ can be
  extended to a continuous selection $f:\mathscr{F}(Z)\to Z$. It now
  remains to observe that the sequence $S_n=\{s_0,\dots, s_n\}$,
  $n<\omega$, is $\tau_V$-convergent to
  $\overline{\bigcup_{n<\omega}S_n}=\overline{S}=S$. Hence,
  $f(S)=\lim_{n\to\infty}f(S_n)= \lim_{n\to\infty} s_n$ which is
  impossible because $S$ is discrete.
\end{proof}

Propositions \ref{pr:extension-1} and \ref{proposition-ext-sel-v25:1}
imply the following characterisation of countably compact totally
disconnected spaces.

\begin{corollary}
  \label{corollary-ext-sel-v25:1}
  Let $X$ be a space with $\sel[\mathscr{F}(X)]\neq\emptyset$.  Then
  $X$ is countably compact and totally disconnected if and only if for
  each discrete set ${Z\in \mathscr{F}(X)}$, each selection for
  $\Sigma(Z)$ can be extended to a continuous selection for
  $\mathscr{F}(X)$.
\end{corollary}

\begin{proof}
  Let $X$ be countably compact and totally disconnected, and
  $Z\in \mathscr{F}(X)$ be a discrete set. Then $Z$ must be a finite
  set because each infinite set in a countably compact space has an
  accumulation point, see \cite[Theorem
  3.10.3]{engelking:89}. Accordingly, since $X$ is also totally
  disconnected, there exists a pairwise-disjoint (i.e.\ discrete)
  collection $\{U_z:z\in Z\}$ of clopen sets with $z\in U_z$, for each
  $z\in Z$. Hence, the property follows from Proposition
  \ref{pr:extension-1}.\smallskip

  Conversely, suppose that $Z\in \mathscr{F}(X)$ is a discrete set
  such that each selection for $\Sigma(Z)$ can be extended to a
  continuous selection for $\mathscr{F}(X)$. Then by Proposition
  \ref{proposition-ext-sel-v25:1}, $Z$ is finite. Accordingly, each
  infinite subsets of $X$ has an accumulation point. Hence, for the
  same reason as before (see again \cite[Theorem
  3.10.3]{engelking:89}), $X$ must be countably compact. Moreover, by
  Corollary \ref{corollary-ext-sel-v5:1}, $X$ is also totally
  disconnected.
\end{proof}

Evidently, Corollary \ref{corollary-ext-sel-v25:1} remains valid if
$Z\in \mathscr{F}(X)$ is only assumed to be countable.\medskip

Since each suborderable space is Tychonoff (being normal), it follows
from Theorems \ref{theorem-ext-sel-v21:1} and
\ref{theorem-ext-sel-v37:1} that for a regular space $X$ with
$\sel[\mathscr{F}(X)]\neq \emptyset$, the hyperspace
selection-extension property in Corollary
\ref{corollary-ext-sel-v25:1} is equivalent to $X$ being a totally
disconnected pseudocompact space and, in particular, such a space must
be also ultranormal. Thus, we have the following consequence which is
complementary to Corollary \ref{corollary-ext-sel-v25:1}. 

\begin{corollary}
  \label{corollary-ext-sel-v40:1}
  Let $X$ be a regular space with $\sel[\mathscr{F}(X)]\neq\emptyset$.
  Then $X$ is countably compact and zero-dimensional if and only
  if for each discrete set ${Z\in \mathscr{F}(X)}$, each selection for
  $\Sigma(Z)$ can be extended to a continuous selection for
  $\mathscr{F}(X)$.
\end{corollary}

There are non-regular weakly orderable countably compact spaces. For
instance, let $X$ be the space obtained from the ordinal space
$\omega_1+1$ by making the set of all countable limit ordinals closed
in $X$, see \cite[3.10.B]{engelking:89}. Since the identity map
$h:X\to \omega_1+1$ is continuous, $X$ is weakly orderable with
respect to the usual linear order ``$\leq$'' on $\omega_1+1$. In
particular, each nonempty subset of $X$ has a first $\leq$-element
and, therefore, $\mathscr{F}(X)$ has a continuous selection, see
Theorem \ref{theorem-hsp-ver-5:7}. Moreover, $X$ is totally
disconnected and satisfies the selection-extension
property in Corollary \ref{corollary-ext-sel-v40:1}, but is not
zero-dimensional because it is not regular.

\section{Extension of Selections and Closed Sets}
\label{sec:extens-hypersp-over-1}

The Vietoris hyperspace $\mathscr{F}(X)$ is partially ordered with
respect to the usual set-theoretic inclusion. To state our first
observation, it will be convenient to consider the following special
subsets of $\mathscr{F}(X)$. We shall say that
$\mathscr{D}\subset \mathscr{F}(X)$ is a \emph{hyperspace-ideal} in
$\mathscr{F}(X)$ if $[X]^1=\mathscr{F}_1(X)\subset \mathscr{D}$ and
$S\in \mathscr{D}$ for every $S\in \mathscr{F}(X)$ such that
$S\subset T$ for some $T\in \mathscr{D}$.  Let us remark that in some
sources, such a subset in a partially ordered set is called a
\emph{weak ideal}.  Evidently, all collections defined in
\eqref{eq:ext-sel-v15:1} are hyperspace-ideals in
$\mathscr{F}(X)$.\medskip

For a hyperspace-ideal $\mathscr{D}\subset \mathscr{F}(X)$ and
$Z\in \mathscr{F}(X)$, let 
\begin{equation}
  \label{eq:ext-sel-v23:1}
  \mathscr{D}(Z)=\{S\in \mathscr{D}: S\subset Z\}=\{S\cap Z: S\in
  \mathscr{D}\ \text{and}\ 
  S\cap Z\neq \emptyset\}.
\end{equation}
Accordingly, $\mathscr{D}(Z)$ is a hyperspace-ideal in
$\mathscr{F}(Z)$. Moreover, $\mathscr{D}(X)=\mathscr{D}$ which gives a
natural pair $(\mathscr{D}(Z),\mathscr{D}(X))$ of subsets of
$\mathscr{F}(X)$. 

\begin{theorem}
  \label{theorem-ext-sel-v46:1}
  Let $X$ be a space and $\mathscr{D}\subset \mathscr{F}(X)$ be a
  hyperspace-ideal such that $\mathscr{F}_2(X)\subset \mathscr{D}$ and
  for every $Z\in \mathscr{F}(X)$, each continuous selection for
  $\mathscr{D}(Z)$ can be extended to a continuous selection
  $\mathscr{D}=\mathscr{D}(X)$. Then $X$ is both normal and strongly
  zero-dimensional, i.e.\ $X$ is ultranormal.
\end{theorem}

\begin{proof}
  Take disjoint sets $A_1,A_2\in \mathscr{F}(X)$ such that there are
  points ${p,q\in X}$ with $p,q\notin A_1\cup A_2$. Also, take
  selections $g_i\in \sel[\mathscr{D}(A_i)]$, $i=1,2$, and set
  $Z=A_1\cup A_2\cup\{p,q\}\in \mathscr{F}(Z)$.  Next, define a
  selection $g:\mathscr{D}(Z)\to Z$ by letting for
  $S\in \mathscr{D}(Z)$ that
  \[
    g(S)=
    \begin{cases}
      g_1(S\cap A_1) &\text{if $S\cap A_1\neq \emptyset$,} \\
      g_2(S\cap A_2) &\text{if $S\subset A_2\cup\{q\}$ and $S\cap
        A_2\neq \emptyset$,}\\ 
      p &\text{if $p\in S\subset \{p\}\cup A_2$,}\\
      q &\text{if $S=\{p,q\}$}.
    \end{cases}
  \]
  Continuity of $g$ follows easily from the fact that $A_1$, $A_2$,
  $\{p\}$ and $\{q\}$ form a clopen partition of $Z$. In fact, the
  important property of this selection is that $A_1<_g A_2\cup\{p,q\}$
  and $p<_g A_2<_g q<_g p$. By condition, $g$ can be extended to a
  selection $h\in \sel[\mathscr{D}]$. In particular,
  $h\uhr \mathscr{F}_2(X)$ is a continuous weak selection for $X$
  because $\mathscr{F}_2(X)\subset \mathscr{D}$, so the relation
  $\leq_h$ is an extension of the relation $\leq_g$. Moreover,
  $U=(p,q)_{\leq_h}=[p,q]_{\leq_h}$ is a clopen set with
  $A_2\subset U$. Since $U\subset (p,\to)_{\leq_h}$ and
  $A_1\subset (\gets,p)_{\leq_h}$, we also have that
  $A_1\subset X\setminus U$. The proof is complete.
\end{proof}

Since each normal space is $\omega$-collectionwise normal, we have the
following question which is complementary to Question
\ref{question-ext-sel-v11:1}.

\begin{question}
  \label{question-ext-sel-v11:2}
  Let $X$ be a space such that for every $Z\in \mathscr{F}(X)$, each
  continuous weak selection for $Z$ can be extended to a continuous
  weak selection for $X$.  Then, is it true that $X$ is collectionwise
  normal?
\end{question}

Another natural question is about the possible inverse of Theorem
\ref{theorem-ext-sel-v46:1}.

\begin{question}
  \label{question-ext-sel-v11:3}
  Let $X$ be a space with a continuous weak selection. If $X$ is
  ultranormal, then is it true that for every $Z\in \mathscr{F}(X)$,
  each continuous weak selection for $Z$ can be extended to a
  continuous weak selection for $X$? What about if $X$ is ultraparacompact?
\end{question}

The answer is ``yes'' for metrizable spaces, where we have a stronger
result. To this end, let us recall that a set-valued mapping
$\varphi:Y\to \mathscr{F}(X)$ is \emph{lower semi-continuous}, or
l.s.c., if the set
$\varphi^{-1}(U)=\{y\in Y:\varphi(y)\cap U\neq\emptyset\}$ is open in
$Y$ for every open $U\subset X$. A map $f:Y\to X$ is a
\emph{selection} for $\varphi$ if $f(y)\in\varphi(y)$ for every
$y\in Y$. The following selection theorem was obtained by Michael in
\cite{michael:56}.

\begin{theorem}[\cite{michael:56}]
  \label{th:set-valued-selections-1}
  Let $Y$ be an ultraparacompact space and $X$ be a metrizable
  space. Then each l.s.c.\ mapping $\varphi:Y\to \mathscr{C}(X)$ has a
  continuous selection. If moreover $X$ is completely metrizable, then
  the same is valid for every l.s.c.\ mappings
  $\varphi:Y\to \mathscr{F}(X)$.
\end{theorem}

This implies the following result regarding Question
\ref{question-ext-sel-v11:3}.

\begin{theorem}
  \label{theorem-ext-sel-v11:1}
  For a metrizable space $X$, the following are equivalent\textup{:}
  \begin{enumerate}
  \item\label{item:ext-sel-v11:1} For every $\tau_V$-closed set
    $\mathscr{G}\subset \mathscr{C}(X)$, each continuous selection for
    $\mathscr{G}$ can be extended to a continuous selection for
    $\mathscr{C}(X)$.
  \item\label{item:ext-sel-v11:2} For every $Z\in \mathscr{F}(X)$,
    each continuous weak selection for $Z$ can be extended to a
    continuous weak selection for $X$.
  \item\label{item:ext-sel-v11:3} $X$ is totally disconnected
    and orderable.  
  \end{enumerate}
\end{theorem}

\begin{proof}
  The implication
  \ref{item:ext-sel-v11:1}$\implies$\ref{item:ext-sel-v11:2} is
  trivial because $\mathscr{G}=\mathscr{F}_2(Z)\subset \mathscr{C}(X)$
  is a $\tau_V$-closed set, while
  \ref{item:ext-sel-v11:2}$\implies$\ref{item:ext-sel-v11:3} follows
  from Theorems \ref{theorem-hsp-ver-10:9} and
  \ref{theorem-ext-sel-v46:1}. To show that
  \ref{item:ext-sel-v11:3}$\implies$\ref{item:ext-sel-v11:1}, take a
  $\tau_V$-closed set $\mathscr{G}\subset \mathscr{C}(X)$ and a
  selection ${g\in \sel[\mathscr{G}]}$. Next, define a set-valued
  mapping $\varphi:\mathscr{C}(X)\to \mathscr{C}(X)$ by
  \[
  \varphi(S)=
  \begin{cases}
    \{g(S)\}\ &\text{if $S\in \mathscr{G}$,\ \ and}\\
    S\ &\text{if $S\notin \mathscr{G}$.}
  \end{cases}
  \]
  Then $\varphi$ is l.s.c.\ with respect to the Vietoris topology on
  $\mathscr{C}(X)$, see \cite[Examples 1.3 and 1.3*]{MR0077107}. To
  show that $\varphi$ has a continuous selection, using Theorem
  \ref{theorem-hsp-ver-10:9}, take a non-Archimedean metric $d$ on $X$
  compatible with the topology of $X$. Next, consider the Hausdorff
  metric $H(d)$ on $\mathscr{C}(X)$:
  \[
    H(d)(S,T)=\sup_{x\in X}\big|d(x,S)-d(x,T)\big|,\quad
    S,T\in \mathscr{C}(X).
  \]
  It is well known that the Hausdorff topology on $\mathscr{C}(X)$
  generated by $H(d)$ coincides with the Vietoris topology
  $\tau_V$. Moreover, $H(d)$ is also a non-Archimedean metric, hence
  $\mathscr{C}(X)$ is strongly zero-dimensional as well. Accordingly,
  by Theorem \ref{th:set-valued-selections-1}, $\varphi$ has a
  continuous selection $f$. Clearly, $f$ is a continuous selection for
  $\mathscr{C}(X)$ extending $g$, which completes the proof.
\end{proof}

Theorem \ref{theorem-ext-sel-v11:1} suggests some related
questions. Regarding non-Archimedean metrics, the strong triangular
inequality goes back to works of J. K\"ursch\'ak \cite{MR1580869} and
A. Ostrowski \cite{MR1555153} on algebraic number theory; for general
metric spaces, see F. Hausdorff \cite{hausdorff:34}. Generalising this
property to arbitrary topological spaces, A.\,F. Monna
\cite{MR0035982,MR0035983} defined a space $X$ to be
\emph{non-Archimedean} if it has a base such that for any two members
$B_1$ and $B_2$ of this base with $B_1\cap B_2\neq\emptyset$, we have
either $B_1\subset B_2$ or $B_2\subset B_1$. A base with this property
is said to be \emph{of rank} 1, see \cite{MR0152989, MR0296881} and
\cite{nagata:62}. Non-Archimedean spaces were defined earlier by
D. Kurepa \cite{zbMATH03023961,zbMATH03027256} under the name of
spaces with \emph{ramified bases}. A space is said to have a
\emph{ramified basis} (or being an \emph{$R$-space}), if it has a
basis which is a tree with respect to the reverse inclusion
``$\supset$''.  Recall that a partially ordered set $(T,\leq)$ is a
\emph{tree} if for every $t\in T$, the set $\{s\in T: s \leq t\}$ is
well-ordered. In Kurepa's terminology (see \cite{zbMATH03023961}), a
tree was called ``un tableau ramifi\'e''. It was shown by P. Papi\v{c}
\cite{MR0054940,MR0073964} that each non-Archimedean space is
ultranormal; and that a space $X$ has a compatible non-Archimedean
metric if and only if $X$ is a metrizable non-Archimedean space.

\begin{question}
  \label{question-ext-sel-v19:1}
  Let $X$ be a non-Archimedean space with a continuous weak
  selection. Then, is it true that for every $Z\in \mathscr{F}(X)$,
  each continuous weak selection for $Z$ can be extended to a
  continuous weak selection for $X$?
\end{question}

Regarding the role of strong zero-dimensionality, Theorems
\ref{theorem-hsp-ver-10:9} and \ref{theorem-ext-sel-v11:1} also
suggest the following question.

\begin{question}
  \label{question-ext-sel-v11:4}
  Let $X$ be a zero-dimensional metrizable space with a continuous
  weak selection. Then, is it true that for every
  $Z\in \mathscr{F}(X)$, each continuous weak selection for $Z$ can be
  extended to a continuous weak selection for $X$?
\end{question}

Evidently, by Theorems \ref{theorem-hsp-ver-10:9},
\ref{theorem-ext-sel-v46:1} and \ref{theorem-ext-sel-v11:1}, the
answer to Question \ref{question-ext-sel-v11:4} is in the affirmative
precisely when $X$ is strongly zero-dimensional. Thus, Question
\ref{question-ext-sel-v11:4} is actually equivalent to the following
question stated in \cite[Question 2]{MR3430989}.

\begin{question}[\cite{MR3430989}]
  \label{question-ext-sel-v19:2}
  Let $X$ be a zero-dimensional metrizable space with a continuous
  weak selection. Then, is it true that $X$ is strongly
  zero-dimensional?
\end{question}

For a family $\{\mathscr{T}_\alpha: \alpha\in \mathscr{A}\}$ of
topologies on a set $X$, the \emph{supremum topology}
$\bigvee_{\alpha\in \mathscr{A}}\mathscr{T}_\alpha$ is the smallest
topology on $X$ which contains all topologies $\mathscr{T}_\alpha$,
${\alpha\in \mathscr{A}}$. A space $(X,\mathscr{T})$ is called a
\emph{continuous} \emph{weak selection space}
\cite{garcia-nogura-2011a}, abbreviated a \emph{CWS-space}, if
$\mathscr{T}=\bigvee_{\sigma\in \Sigma}\mathscr{T}_\sigma$ for some
collection $\Sigma\subset \sel[\mathscr{F}_2(X)]$. Some basic
properties of these spaces, also several examples, were provided in
\cite{garcia-nogura-2011a}. In particular, it was shown in
\cite[Theorem 2.7]{garcia-nogura-2011a}, see also \cite{MR3430989},
that each zero-dimensional space $X$ with a continuous weak selection
is a CWS-space.  Thus, Question \ref{question-ext-sel-v11:4} is also a
special case of the following question stated in \cite[Question
1]{MR3430989}.

\begin{question}[\cite{MR3430989}]
  \label{question-ext-sel-v11:5}
  Let $X$ be a metrizable CWS-space. Then, is it true that $X$ is
  suborderable? 
\end{question}

Another aspect is the role of metrizability in such hyperspace
selection-exten\-sion properties, it brings the following general
question.

\begin{question}
  \label{question-ext-sel-v13:1}
  Let $X$ be a totally disconnected suborderable space. Then, is it
  true that for every $Z\in \mathscr{F}(X)$, each continuous weak
  selection for $Z$ can be extended to a continuous weak selection for
  $X$?
\end{question}

A good starting point to consider Question
\ref{question-ext-sel-v13:1} is when $X$ is the first uncountable
ordinal $\omega_1$, or in the compact case when $X=\omega_1+1$.

\begin{question}
  \label{question-ext-sel-v13:2}
  Let $\omega_1\subset X\subset \omega_1+1$. Then, is it true that for
  every ${Z\in \mathscr{F}(X)}$, each continuous weak selection for $Z$
  can be extended to a continuous weak selection for $X$?
\end{question}

Evidently, the answer is ``yes'' if $Z$ is a countable set. The other
case is when $Z$ is uncountable. A subset of $\omega_1$ is uncountable
iff it is cofinal and unbounded in $\omega_1$. Hence, the remaining
case is when $Z\subset X$ is a \emph{club} set. It is well known that
$Z\subset \omega_1$ is a club set if and only if there is a continuous
order-preserving bijection from $\omega_1$ onto $Z$.  Accordingly, in
this case, a continuous weak selection for a club set
$Z\subset \omega_1$ is the same as a continuous weak selection for
$\omega_1$.\medskip

The relationship between continuous weak selections for $\omega_1$ and
those for $\omega_1+1$ is essentially known. Namely, since
$\beta \omega_1=\omega_1+1$, Theorem \ref{theorem-ext-sel-v13:1}
implies that Question \ref{question-ext-sel-v13:2} is reduced only to
the case of $X=\omega_1$.

\begin{corollary}
  \label{corollary-ext-sel-v13:1}
  Each continuous weak selection for $\omega_1$ can be extended to a
  continuous weak selection for $\omega_1+1$. 
\end{corollary}

In fact, it is also enough to consider Question
\ref{question-ext-sel-v13:2} only in the case of
$X=\omega_1+1$. Indeed, let $Z\subset \omega_1$ be a club set and
$Y=Z\cup\{\omega_1\}$ be the closure of $Z$ in
$\omega_1+1=\beta \omega_1$. Then $\beta Z=Y$ because $\omega_1$ is
normal, see \cite[Corollary 3.6.8]{engelking:89}. Moreover, $Z$ is
pseudocompact being a sequentially compact Tychonoff space, see
Theorem \ref{theorem-ext-sel-v21:1}. Hence, by Theorem
\ref{theorem-ext-sel-v13:1}, each continuous weak selection for $Z$
can be extended to a continuous weak selection for $Y=\beta Z$. Thus,
if the answer to Question \ref{question-ext-sel-v13:2} is ``yes'' for
$X=\omega_1+1$, then it is also ``yes'' for $X=\omega_1$.\medskip

Regarding the role of closed sets in such hyperspace
selection-extension problems, it is interesting to consider the
following general question which is complementary to Question
\ref{question-ext-sel-v13:3}. 

\begin{question}
  \label{question-ext-sel-v50:1}
  Let $X$ be a space such that for every subspace ${Z\subset X}$, each
  continuous weak selection for $Z$ can be extended to a continuous
  weak selection for $X$. Then what can one say about $X$? Does there
  exist a non-discrete space $X$ with this property?
\end{question}

Evidently, if $X$ is as in Question \ref{question-ext-sel-v50:1}, then
it doesn't contain any non-trivial convergent sequence, see
Proposition \ref{proposition-ext-sel-v13:2}. In fact, all results
regarding the case of a discrete subspace $Z\subset X$ remain valid
for a space $X$ as in this question.  Here is a related observation
where $Z\subset X$ is not required to be a discrete subspace. A point
$p\in Y$ in a space $Y$ is called \emph{butterfly} (or a
\emph{b-point}) \cite{Sapirovskii1975} if
$\overline{F\setminus\{p\}}\cap \overline{G\setminus\{p\}}=\{p\}$ for
some closed sets $F,G\subset Y$. Following the idea of cut points in
connected spaces, $p\in Y$ was called a \emph{cut point} of a space
$Y$ \cite{gutev-nogura:03a} if $Y\setminus\{p\}=U\cup V$, where $U$
and $V$ are disjoint sets such that
$\overline{U}\cap \overline{V}=\{p\}$. Points with this property were
also introduced in \cite{Dow2008}, where they were called
\emph{tie-points}. In some sources, cut points (equivalently,
tie-points) are called \emph{strong butterfly points}. Regarding
butterfly points of subspaces, we shall say that $p\in X$ is a
\emph{weak butterfly point} if there are disjoint sets
$S,T\subset X\setminus \{p\}$ such that
$p\in \overline{S}\cap \overline{T}$. Evidently, each butterfly point
of some subset $Y\subset X$ is a weak butterfly point of $X$.

\begin{proposition}
  \label{proposition-ext-sel-v14:1}
  Let $X$ be a space such that for every $Z\subset X$, each continuous
  weak selection for $Z$ can be extended to a continuous weak
  selection for $X$. Then $X$ doesn't contain any weak butterfly point. 
\end{proposition}

\begin{proof}
  Assume that $S,T\subset X\setminus\{p,q\}$ are disjoint subsets for
  some points $p,q\in X$ such that
  $p\in \overline{S}\cap \overline{T}$. Next, take selections
  $\eta_S\in\sel[\mathscr{F}_2(S)]$ and
  $\eta_T\in \sel[\mathscr{F}_2(T)]$. Finally, set
  $Z=S\cup\{q\}\cup T$ and define a continuous weak selection $\eta$
  for $Z$ by $\eta\uhr \mathscr{F}_2(S)= \eta_S$,
  $\eta\uhr \mathscr{F}_2(T)= \eta_T$ and $q<_\eta S<_\eta T<_\eta
  q$. Then, by condition, $\eta$ can be extended to a continuous weak
  selection $\sigma$ for $X$. Since $q<_\sigma S$ and
  $p\in \overline{S}$, the continuity of $\sigma$ implies that
  $q<_\sigma p$. Similarly, we also have that $p<_\sigma q$ because
  $T<_\sigma q$ and $p\in \overline{T}$. Clearly, this is impossible.
\end{proof}

We have an interesting consequence of the above result. To this end,
let us recall that a space $X$ is \emph{extremally disconnected} if
$\overline{U}$ is open for every open $U\subset X$. 

\begin{corollary}
  \label{corollary-ext-sel-v50:1}
    Let $X$ be a space such that for every $Z\subset X$, each continuous
  weak selection for $Z$ can be extended to a continuous weak
  selection for $X$. Then $X$ is extremally disconnected. 
\end{corollary}

\begin{proof}
  Assume that $\overline{U}$ is not open for some open set
  $U\subset X$. Then $V=X\setminus\overline{U}$ is not closed and,
  consequently, there is a point $p\in \overline{V}\setminus V$.
  Thus, we also have that $p\in \overline{U}\setminus U$, so $p$ is a
  weak butterfly point of $X$. However, by Proposition
  \ref{proposition-ext-sel-v14:1}, this is impossible.
\end{proof}

We conclude this section considering the same extension problem but
with respect to arbitrary continuous selections. For such selections,
we have the following result which is complementary to Theorem
\ref{theorem-ext-sel-v11:1}.

\begin{theorem}
  \label{theorem-ext-sel-v48:2}
  For a metrizable space $X$, the following are equivalent\textup{:}
  \begin{enumerate}
  \item\label{item:ext-sel-v21:1} $X$ is totally disconnected and
    compact.
  \item\label{item:ext-sel-v21:2} For every $\tau_V$-closed set
    $\mathscr{G}\subset \mathscr{C}(X)$, each continuous selection for
    $\mathscr{G}$ can be extended to a continuous selection
    $\mathscr{F}(X)$.
  \item\label{item:ext-sel-v21:3} For every countable discrete
    $Z\in \mathscr{F}(X)$, each selection for $\Sigma(Z)$ can
    be extended to a continuous selection for $\mathscr{F}(X)$.
  \end{enumerate}
\end{theorem}

\begin{proof}
  A separable metrizable space $X$ is totally disconnected if and only
  if there exists a continuous injective map of $X$ into the Cantor
  set $\mathbf{C}=\{0,1\}^{\omega}$, see Kuratowski
  \cite{zbMATH03030217} and \cite[Theorem 3,
  p. 148]{MR0259835}. Hence, the implication
  \ref{item:ext-sel-v21:1}$\implies$\ref{item:ext-sel-v21:2} is a
  special case of the implication
  \ref{item:ext-sel-v11:3}$\implies$\ref{item:ext-sel-v11:1} in
  Theorem \ref{theorem-ext-sel-v11:1}.  The implication
  \ref{item:ext-sel-v21:2}$\implies$\ref{item:ext-sel-v21:3} is
  trivial because $\mathscr{G}=\Sigma(Z)\subset \mathscr{C}(X)$ is a
  $\tau_V$-closed set for every discrete $Z\in
  \mathscr{F}(X)$. Finally,
  \ref{item:ext-sel-v21:3}$\implies$\ref{item:ext-sel-v21:1} follows
  from Corollary \ref{corollary-ext-sel-v25:1} because each countably
  compact metrizable space is compact, see e.g.\ \cite[Theorem
  4.1.17]{engelking:89}.
\end{proof}

Theorem \ref{theorem-ext-sel-v48:2} implies the following interesting 
consequence.

\begin{corollary}
  \label{corollary-ext-sel-v21:4}
  Let $X$ be a metrizable space such that for every
  $Z\in \mathscr{F}(X)$, each continuous selection for
  $\mathscr{C}(Z)$ can be extended to a continuous selection for
  $\mathscr{F}(X)$. Then $X$ is a closed subset of the Cantor set
  $\mathbf{C}$ and
  \begin{equation}
    \label{eq:ext-sel-v21:1}
    \sel[\mathscr{F}(X)]=\left\{h\uhr \mathscr{F}(X): h\in
      \sel[\mathscr{F}(\mathbf{C})]\right\}. 
  \end{equation}
\end{corollary}

This brings the following natural question.

\begin{question}
  \label{question-ext-sel-v26:3}
  Let $X$ be a metrizable space such that for every compact
  $Z\subset X$, each continuous selection for
  $\mathscr{F}(Z)=\mathscr{C}(Z)$ can be extended to a continuous
  selection for $\mathscr{F}(X)$.  Then what can one say about $X$?
\end{question}

Evidently, by Theorem \ref{theorem-ext-sel-v48:2}, each totally
disconnected compact metrizable $X$ is as in Question
\ref{question-ext-sel-v26:3}. However, there are also non-compact
metrizable spaces with this hyperspace selection-extension
property. For instance, each discrete space has the
selection-extension property in Question
\ref{question-ext-sel-v26:3}. Namely, as mentioned before, each
discrete space $X$ has a continuous selection for $\mathscr{F}(X)$,
see page \pageref{page-ref-discrete}. Moreover, each nonempty compact
$Z\subset X$ is finite, so $\mathscr{C}(Z)=\Sigma(Z)$ is
$\tau_V$-clopen in $\mathscr{F}(X)$ and the selection extension
property follows from Proposition
\ref{proposition-ext-sel-v32:2}.\medskip  

If $X$ is as in Question \ref{question-ext-sel-v26:3}, then by Theorem
\ref{theorem-ext-sel-v4:1}, $X$ is totally disconnected. In
particular, each compact subset of $X$ is (strongly)
zero-dimensional. Accordingly, Theorem \ref{theorem-ext-sel-v11:1}
brings another natural question.

\begin{question}
  \label{question-ext-sel-v26:5}
  Let $X$ be a strongly zero-dimensional completely metrizable space
  Is it true that for every $Z\in \mathscr{F}(X)$, each continuous
  selection for $\mathscr{F}(Z)$ can be extended to a continuous
  selection for $\mathscr{F}(X)$? What about if $Z$ is compact?
\end{question}

Regarding Question \ref{question-ext-sel-v26:5}, let us remark that we
cannot apply Theorem \ref{th:set-valued-selections-1} as we did in
Theorem \ref{theorem-ext-sel-v11:1}. Namely, the Vietoris hyperspace
$\mathscr{F}(X)$ is normal precisely when $X$ is compact, see Keesling
\cite{keesling:70} and Velichko \cite{velichko:75}. Let us also remark
that the conditions on the space $X$ in Question
\ref{question-ext-sel-v26:5} are necessary. Namely, by Theorem
\ref{theorem-ext-sel-v46:1}, such a space $X$ must be strongly
zero-dimensional. Moreover, by Theorem \ref{theorem-ext-sel-v47:2},
each metrizable space $X$ with $\sel[\mathscr{F}(X)]\neq\emptyset$,
is completely metrizable. Finally, by Theorem
\ref{theorem-ext-sel-v26:4}, each completely metrizable strongly
zero-dimensional space $X$ has a continuous selection for
$\mathscr{F}(X)$.\medskip

Another point of view is to consider special ``extreme-like''
selections, similar to those in the connected case. A selection
$f:\mathscr{F}(X)\to X$ is called \emph{monotone}
\cite{gutev-nogura:01a} if $f(S)=f(T)$ for every
$S,T\in \mathscr{F}(X)$, with $f(T)\in S\subset T$. It was shown in
\cite[Proposition 4.7]{gutev-nogura:01a} that each monotone selection
$f$ is continuous, and in \cite[Proposition 4.8]{gutev-nogura:01a}
that $\leq_f$ is a linear order on $X$. Hence, in particular, $X$ is
weakly orderable with respect to $\leq_f$ for each monotone selection
$f$ for $\mathscr{F}(X)$. In fact, it was shown a stronger property in
\cite[Lemma 4.9]{gutev-nogura:01a} that in this case,
$f(S)=\min_{\leq_f} S$ for every $S\in \mathscr{F}(X)$. Thus,
we also have the following question.

\begin{question}
  \label{question-ext-sel-v26:6}
  Let $X$ be a strongly zero-dimensional completely metrizable
  space. Is it true that for every $Z\in \mathscr{F}(X)$, each
  monotone selection for $\mathscr{F}(Z)$ can be extended to a
  continuous (monotone) selection for $\mathscr{F}(X)$?
\end{question}

Here are also two questions for particular spaces.

\begin{question}
  \label{question-ext-sel-v26:7}
  Let $\iR$ be the irrational numbers, $\emptyset\neq Z\subset \iR$ be
  a closed set and $g\in \sel[\mathscr{F}(Z)]$. Is it possible to
  extend $g$ to a continuous selection $f\in\sel[\mathscr{F}(\iR)]$?
\end{question}

\begin{question}
  \label{question-ext-sel-v26:8}
  Let $\emptyset\neq Z\subset \omega_1$ be a closed set and
  $g\in \sel[\mathscr{F}(Z)]$. Is it possible to extend $g$ to a
  continuous selection $f\in\sel[\mathscr{F}(\omega_1)]$?
\end{question}

\section{Extension of Selections and Dense Hyperspaces}
\label{sec:extens-probl-hypersp}

In this section, we consider the hyperspace \emph{selection-extension
  property} for pairs $(\mathscr{G},\mathscr{H})$ in the particular
case when $\mathscr{G}$ is $\tau_V$-dense in $\mathscr{H}$, see
Definition \ref{definition-ext-sel-v4:1}.

\begin{question}
  \label{question-ext-sel-v32:1}
  Does there exist a space $X$ and subfamilies
  $\mathscr{G}\subsetneq \mathscr{H}\subset \mathscr{F}(X)$ such that
  $\mathscr{G}$ is $\tau_V$-dense in $\mathscr{H}$ and the pair
  $(\mathscr{G},\mathscr{H})$ has the selection-extension property?
\end{question}

We will consider the above question in the setting of the basic
hyperspace-ideals
$\Sigma(X)\subset \mathscr{C}(X)\subset \mathscr{F}(X)$, namely when
\begin{align}
  \label{eq:ext-sel-v26:2}
  \sel[\Sigma(X)]&=\big\{f\uhr \Sigma(X): f\in
                   \sel[\mathscr{C}(X)]\big\}\neq\emptyset,\\
  \label{eq:ext-sel-v26:1}
  \sel[\mathscr{C}(X)]&=\big\{f\uhr \mathscr{C}(X): f\in
  \sel[\mathscr{F}(X)]\big\}\neq\emptyset \quad \text{and} \\
  \label{eq:ext-sel-v26:3}
  \sel[\Sigma(X)]&=\big\{f\uhr \Sigma(X): f\in
                                 \sel[\mathscr{F}(X)]\big\}\neq\emptyset.   
\end{align}

Evidently, \eqref{eq:ext-sel-v26:2} and \eqref{eq:ext-sel-v26:1}
together imply \eqref{eq:ext-sel-v26:3}. Conversely,
\eqref{eq:ext-sel-v26:3} implies both \eqref{eq:ext-sel-v26:2} and
\eqref{eq:ext-sel-v26:1} which is based on the following simple
observation.

\begin{proposition}
  \label{proposition-ext-sel-v22:1}
  Let $X$ be a space,
  $\Sigma(X)\subset \mathscr{H}\subset \mathscr{F}(X)$ and
  $g,h\in \sel[\mathscr{H}]$ be such that
  $g\uhr \Sigma(X)= h\uhr \Sigma(X)$. Then $g=h$.
\end{proposition}

\begin{proof}
  Follows from the fact that $X$ is Hausdorff and $\Sigma(X)$ is
  $\tau_V$-dense in $\mathscr{H}$. 
\end{proof}

Another related property is the following consequence of Proposition
\ref{proposition-ext-sel-v32:2}.

\begin{proposition}
  \label{proposition-ext-sel-v32:1}
  If a space $X$ satisfies some of the properties
  \eqref{eq:ext-sel-v26:2}, \eqref{eq:ext-sel-v26:1} or
  \eqref{eq:ext-sel-v26:3}, then so does any nonempty clopen subset of
  $X$.
\end{proposition}

If $\mathscr{C}(X)=\mathscr{F}(X)$, then the selection-extension
property in \eqref{eq:ext-sel-v26:1} is trivial. Furthermore, in this
case, $X$ is compact because $X\in \mathscr{F}(X)$. This brings the
following special case of Question \ref{question-ext-sel-v32:1}.

\begin{question}
  \label{question-ext-sel-v32:2}
  Does there exist a space $X$ which satisfies
  \eqref{eq:ext-sel-v26:1}, but is not compact?
\end{question}

If a connected space $X$ satisfies \eqref{eq:ext-sel-v26:1}, then it
must be compact. This follows from Corollary
\ref{corollary-hsp-ver-14:1} and Theorem \ref{theorem-hsp-ver-17:1},
see also Theorem \ref{theorem-ext-sel-v15:2}. In fact, we have the
following more general result.

\begin{proposition}
  \label{proposition-ext-sel-v34:2}
  If $X$ is weakly orderable and satisfies
  \eqref{eq:ext-sel-v26:1}, then it is compact. 
\end{proposition}

\begin{proof}
  Take a compatible linear order $\leq$ on $X$, and a set
  $T\in \mathscr{F}(X)$. Next, using Theorem
  \ref{theorem-hsp-ver-5:7}, consider the selection
  $g\in \sel[\mathscr{C}(X)]$ defined by $g(S)=\min_\leq S$ for every
  $S\in \mathscr{C}(X)$. Then by \eqref{eq:ext-sel-v26:1}, $g$ can be
  extended to a selection $f\in \sel[\mathscr{F}(X)]$. We are going to
  show that $p=f(T)$ is the first $\leq$-element of $T$. To this end,
  assume that $q<p$ for some $q\in T$. Since $f$ is continuous,
  $(q,\to)_\leq$ is open and $p\in (q,\to)_\leq$, there exists a
  finite family $\mathscr{U}$ of open subsets of $X$ such that
  $T\in \langle\mathscr{U}\rangle$ and
  $f(\langle\mathscr{U}\rangle)\subset (q,\to)_\leq$. Finally, take
  a finite set $S\in \langle\mathscr{U}\rangle$ with $q\in S$. Then
  $f(S)\in (q,\to)_\leq$ and $f(S)=g(S)=\min_\leq S\leq q$ because
  $q\in S$. Clearly, this is impossible. Thus, $p=\min_\leq
  T$. Evidently, using the inverse linear order, $T$ also has a
  last $\leq$-element. According to Proposition
  \ref{proposition-ext-sel-v34:1}, $X$ is compact and orderable.
\end{proof}

It was shown in \cite[Lemma 7.4]{michael:51} that each locally
connected spaces with a continuous weak selection is weakly
orderable. Hence, we also have the following immediate consequence of
Proposition \ref{proposition-ext-sel-v34:2}.

\begin{corollary}
  \label{corollary-ext-sel-v32:1}
  If $X$ is locally connected and satisfies \eqref{eq:ext-sel-v26:1},
  then it is compact. 
\end{corollary}

This brings the following question which is complementary to Question
\ref{question-ext-sel-v32:2}.

\begin{question}
  \label{question-ext-sel-v22:1}
  Let $X$ be a space satisfying \eqref{eq:ext-sel-v26:1}. Then, is it
  true that each component of $X$ is compact?
\end{question}

The answer is ``yes'' when $X$ is locally compact.

\begin{corollary}
  \label{corollary-ext-sel-v34:1}
  If $X$ is locally compact and satisfies \eqref{eq:ext-sel-v26:1},
  then each point of $X$ is contained in a clopen compact subset of
  $X$. 
\end{corollary}

\begin{proof}
  According to \cite[Corollaries 5.2 and 5.3]{gutev:07b}, each point
  of $X$ is contained in a clopen orderable subset of $X$. Since $X$
  satisfies \eqref{eq:ext-sel-v26:1}, the property follows from
  Propositions \ref{proposition-ext-sel-v32:1} and
  \ref{proposition-ext-sel-v34:2}.  
\end{proof}

This brings another related question.

\begin{question}
  \label{question-ext-sel-v34:1}
  If $X$ is locally compact and satisfies \eqref{eq:ext-sel-v26:1},
  then is it true that $X$ is compact?
\end{question}

The answer is ``yes'' when $X$ is also paracompact.

\begin{corollary}
  \label{corollary-ext-sel-v34:2}
  If $X$ is a locally compact paracompact space and satisfies
  \eqref{eq:ext-sel-v26:1}, then it is compact.
\end{corollary}

\begin{proof}
  According to \cite[Theorem 5.1]{gutev:07b}, each locally compact
  paracompact space with a continuous weak selection is weakly
  orderable. Hence, the property follows from Proposition
  \ref{proposition-ext-sel-v34:2}. 
\end{proof}

Evidently, Corollary \ref{corollary-ext-sel-v34:2} covers the case
when $X$ is a locally compact Lindel\"of space because such a space is
also paracompact, see \cite[Theorem 5.1.2]{engelking:89}. This
suggests the following further question.

\begin{question}
  \label{question-ext-sel-v34:2}
  If $X$ is Lindel\"of and satisfies \eqref{eq:ext-sel-v26:1}, then is
  it true that $X$ is compact?
\end{question}

Every separable paracompact space is Lindel\"of, and for metrizable
spaces separability is equivalent to the Lindel\"of property. Hence,
the answer is ``yes'' in the following special case.

\begin{corollary}
  \label{corollary-ext-sel-v34:3}
  If $X$ is a separable metrizable space which satisfies
  \eqref{eq:ext-sel-v26:1}, then it is compact. 
\end{corollary}

\begin{proof}
  According to \cite[Theorem 1.1]{gutev:07a}, see also \cite[Theorem
  4.1]{MR3430989}, $X$ is weakly orderable because it has a continuous
  weak selection. Hence, just like before, Proposition
  \ref{proposition-ext-sel-v34:2} implies the property. 
\end{proof}

This brings the following separate question.

\begin{question}
  \label{question-ext-sel-v34:3}
  If $X$ is a metrizable space and satisfies \eqref{eq:ext-sel-v26:1},
  then is it true that $X$ is compact?
\end{question}

The answer is ``yes'' when $X$ is zero-dimensional, which is based on
the following considerations.

\begin{theorem}
  \label{theorem-ext-sel-v32:1}
  If $X$ is a space which satisfies \eqref{eq:ext-sel-v26:1}, then
  each locally finite clopen cover of $X$ is finite.
\end{theorem}

\begin{proof}
  Assume that $X$ has an infinite locally finite clopen cover. Then,
  it also has a locally finite clopen cover $\mathscr{U}$ which is
  both infinite and countable. Take an enumeration
  $\mathscr{U}=\{U_n:n<\omega\}$ of this cover. Next, for each
  $S\in \mathscr{C}(X)$, let
  $\kappa(S)=\max\{n<\omega: S\cap U_n\neq \emptyset\}$. Finally, for
  each $n<\omega$, take a selection $g_n\in \sel[\mathscr{C}(U_n)]$
  and define another selection $g:\mathscr{C}(X)\to X$ by
  \begin{equation}
    \label{eq:ext-sel-v22:1}
    g(S)=g_{\kappa(S)}(S\cap U_{\kappa(S)})\quad \text{for every
      $S\in \mathscr{C}(X)$.}
  \end{equation}
  It is easy to verify that $g$ is continuous because each $U_n$,
  $n<\omega$, is clopen. So, by \eqref{eq:ext-sel-v26:1}, $g$ can be
  extended to a selection $f\in\sel[\mathscr{F}(X)]$. To see that this
  is impossible, we will construct a sequence
  $\{y_k:k<\omega\}\subset X$ and a subsequence
  $\left\{U_{n_k}:k<\omega\right\}$ of $\{U_n:n<\omega\}$ such that
  \begin{equation}
    \label{eq:ext-sel-v34:1}
    n_k=\max\{n<\omega: y_k\in U_n\}\quad \text{for every $k<\omega$.}
  \end{equation}
  Briefly, take $y_0\in U_0$, and let $n_0\geq 0$ be as in
  \eqref{eq:ext-sel-v34:1}. Next, take $y_1\in U_{n_0+1}$ and just
  like before, let $n_1>n_0$ be the unique natural number as in
  \eqref{eq:ext-sel-v34:1}. The construction can be carried on by
  induction.\smallskip

  We finish the proof as follows. Since the family
  $\{U_{n_k}:k<\omega\}$ is locally finite, it follows from
  \eqref{eq:ext-sel-v34:1} that $Y=\{y_k:k<\omega\}$ is an infinite
  closed discrete set. For the same reason, if
  $Y_k=\{y_0,\dots, y_k\}$, then $\kappa(Y_k)=n_k$ and by
  \eqref{eq:ext-sel-v22:1}, $f(Y_k)=g(Y_k)=y_k$. Furthermore, the
  sequence $Y_k$, $k<\omega$, is $\tau_V$-convergent to
  $\overline{\bigcup_{k<\omega}Y_n}=Y$. Therefore,
  $f(Y)=\lim_{k\to\infty}f(Y_k)= \lim_{k\to\infty} y_k$ but this
  contradicts the fact that $Y$ is discrete.
\end{proof}
  
A \emph{$\pi$-base} for a space $X$ (called also a \emph{pseudobase},
Oxtoby \cite{oxtoby:60}) is a family $\mathscr{P}$ of open subsets
such that each nonempty open subset of $X$ contains some nonempty
member of $\mathscr{P}$.

\begin{corollary}
  \label{corollary-ext-sel-v34:4}
  Let $X$ be a strongly $\omega$-cwH space which has a clopen
  $\pi$-base. If $X$ satisfies \eqref{eq:ext-sel-v26:1}, then it is
  compact.
\end{corollary}

\begin{proof}
  First of all, let us show that $X$ is countably compact. Assume that
  this fails. Then $X$ has an infinite countable discrete set
  $Z\in \mathscr{F}(X)$, see \cite[Theorem
  3.10.3]{engelking:89}. Hence, since $X$ is strongly $\omega$-cwH,
  there exists a discrete open family $\{V_z:z\in Z\}$ with
  $z\in V_z$, for each $z\in Z$. Accordingly, $X$ also has an infinite
  clopen partition because it has a clopen $\pi$-base. However, by
  Theorem \ref{theorem-ext-sel-v32:1}, this is impossible because $X$
  satisfies \eqref{eq:ext-sel-v26:1}. Thus, $X$ is countably
  compact. Therefore, by Theorem \ref{theorem-ext-sel-v37:1} and
  Proposition \ref{proposition-ext-sel-v34:2}, $X$ is also compact.
\end{proof}

As remarked before, each normal (i.e.\ $\omega$-collectionwise normal)
space is strongly $\omega$-cwH. Moreover, each suborderable space is
both collectionwise normal and countably paracompact. However, there
are many examples of countably paracompact spaces which are not
normal. For instance, $X=(\omega_1 + 1) \times \omega_1$ is even
countably compact, but is not normal and doesn't satisfy
\eqref{eq:ext-sel-v26:1} because it is not compact, see Theorem
\ref{theorem-ext-sel-v37:1} and Proposition
\ref{proposition-ext-sel-v34:2}. Thus, we have the following related
question.

\begin{question}
  \label{question-ext-sel-v37:2}
  Let $X$ be a countably paracompact space with a clopen
  $\pi$-base. If $X$ satisfies \eqref{eq:ext-sel-v26:1}, then is it
  true that $X$ is normal?
\end{question}

Evidently, if the answer is ``yes'', then such a space must be also
compact, as per Corollary \ref{corollary-ext-sel-v34:4}.\medskip

A space $X$ is \emph{weakly paracompact} (called also
\emph{metacompact}) if every open cover of $X$ has a point-finite open
refinement. Each weakly paracompact normal space is countably
paracompact, see \cite[Theorem 5.2.6]{engelking:89} and
\cite[Corollary 6.4]{gutev:11}. Accordingly, we also have the
following question.

\begin{question}
  \label{question-ext-sel-v37:3}
  Let $X$ be a weakly paracompact space with a clopen $\pi$-base. If
  $X$ satisfies \eqref{eq:ext-sel-v26:1}, then is it true that $X$ is
  compact (equivalently, normal)?
\end{question}

Spaces with clopen $\pi$-bases have a very interesting interpretation
in terms of hyperspace selections. The following result was
obtained in \cite[Theorem 2.1]{gutev:05a}.

\begin{theorem}[\cite{gutev:05a}]
  \label{theorem-ext-sel-v26:2}
  Let $X$ be a space with $\sel[\mathscr{F}(X)]\neq \emptyset$.  Then
  $X$ has a clopen $\pi$-base if and only if the set
  $\left\{f(X) : f\in \sel[\mathscr{F}(X)]\right\}$ is dense in $X$.
\end{theorem}

It was further shown in \cite[Corollary 2.3]{gutev:05a}, see also
\cite[Theorem 1.5]{gutev-nogura:00b}, that if $X$ has a clopen
$\pi$-base and $\sel[\mathscr{F}(X)]\neq \emptyset$, then $X$ is
totally disconnected. On the other hand, as commented before, each
normal space is strongly $\omega$-cwH.  Thus, in view of Corollary
\ref{corollary-ext-sel-v34:4}, we also have the following question.

\begin{question}
  \label{question-ext-sel-v26:1}
  Let $X$ be a totally disconnected normal space. If $X$  satisfies 
  \eqref{eq:ext-sel-v26:1}, then is it true that $X$ is compact? What
  about if $X$ is metrizable?
\end{question}

In fact, Question \ref{question-ext-sel-v26:1} is very much related to
\cite[Question 391]{gutev-nogura:06b} and the following more general
question.

\begin{question}[\cite{gutev-nogura:06b}]
  \label{question-ext-sel-v26:2}
  Let $X$ be a totally disconnected space with
  $\sel[\mathscr{F}(X)]\neq\emptyset$. Then, is it true that $X$ has a
  clopen $\pi$-base?
\end{question}

According to Theorem \ref{theorem-ext-sel-v26:2} and \cite[Theorem
6.2]{MR3122363}, the answer to the above question is ``yes'' when $X$
is a countably compact regular space. In fact, it was shown in
\cite[Corollary 6.4]{MR3122363} that for a countably compact regular
space with $\sel[\mathscr{F}(X)]\neq\emptyset$, total disconnectedness
of $X$ is equivalent to zero-dimensionality of $X$. Evidently, by
Theorem \ref{theorem-ext-sel-v37:1}, the result remains valid for
every countably compact regular spaces with a continuous weak
selection because each suborderable totally disconnected space is
zero-dimensional, see page \pageref{page:suborderable-zero}. This
brings the following special case of Question
\ref{question-ext-sel-v26:2}. 

\begin{question}
  \label{question-ext-sel-v40:1}
  Let $X$ be a totally disconnected weakly orderable space with
  $\sel[\mathscr{F}(X)]\neq\emptyset$. Then, is it true that $X$ has a
  clopen $\pi$-base?
\end{question}

The property in \eqref{eq:ext-sel-v26:3} is very similar to that of
\eqref{eq:ext-sel-v26:1} and as remarked before, it implies both
\eqref{eq:ext-sel-v26:2} and \eqref{eq:ext-sel-v26:1}. Hence, all
results obtained for spaces with property \eqref{eq:ext-sel-v26:1}
remain valid for those with property \eqref{eq:ext-sel-v26:3}. To
study the difference between \eqref{eq:ext-sel-v26:1} and
\eqref{eq:ext-sel-v26:3}, we first consider the property in
\eqref{eq:ext-sel-v26:2}. For this property, we have the following two
crucial examples.

\begin{example}
  \label{example-ext-sel-v35:1}
  There exists a suborderable space $X$ which is not discrete and
  $\mathscr{C}(X)=\Sigma(X)$.  For instance, take $X$ to be the
  subspace of $\omega_1+1$ obtained by removing all countable limit
  ordinals. Evidently, by Theorem \ref{theorem-hsp-ver-5:7},
  $\sel[\mathscr{C}(X)]\neq \emptyset$ and $X$ satisfies
  \eqref{eq:ext-sel-v26:2}.
\end{example}

\begin{example}
  \label{example-ext-sel-v35:2}
  The closed unit interval $X=[0,1]$ is compact and orderable, and
  clearly $\Sigma(X)\neq \mathscr{C}(X)$. However, $X$ satisfies
  \eqref{eq:ext-sel-v26:2}, see Theorems \ref{theorem-ext-sel-v15:2}
  and \ref{theorem-hsp-ver-5:8}. In particular, in this case, the pair
  $(\Sigma(X),\mathscr{C}(X))$ gives an affirmative answer to Question
  \ref{question-ext-sel-v32:1}.
\end{example}

In fact, Example \ref{example-ext-sel-v35:2} is valid for any
connected space with a continuous weak selection, which suggests the
following natural question.

\begin{question}
  \label{question-ext-sel-v35:1}
  Let $X$ be a locally connected space with a continuous weak
  selection. Then, is it true that $X$ satisfies
  \eqref{eq:ext-sel-v26:2}? 
\end{question}

If a space $X$ has a clopen partition $\mathscr{P}$ with
$\sel[\mathscr{C}(P)]\neq \emptyset$, $P\in \mathscr{P}$, then
$\sel[\mathscr{C}(X)]\neq \emptyset$. This brings the following more
general question.

\begin{question}
  \label{question-ext-sel-v35:2}
  Let $X$ be a space which has a clopen partition $\mathscr{P}$ such
  that each $P\in \mathscr{P}$ satisfies
  \eqref{eq:ext-sel-v26:2}. Then, is it true that $X$ also satisfies
  \eqref{eq:ext-sel-v26:2}?
\end{question}

If $X$ is a space which satisfies \eqref{eq:ext-sel-v26:2} and each
point of $X$ has a countable clopen base, then $X$ must be a discrete
space. In fact, this is true in the following more general
situation. A point $p\in X$ is \emph{countably-approachable}
\cite{gutev:05a} if it is either isolated or has a countable clopen
base in $\overline{U}$ for some open set ${U\subset X\setminus\{p\}}$
with $\overline{U}=U\cup\{p\}$. One can easily see that a non-isolated
point $p\in X$ is countably-approachable if and only if $X$ contains a
pairwise-disjoint sequence $\{S_n\}$ of nonempty clopen sets such that
$p\notin S_n$, $n < \omega$, and $\{S_n\}$ is convergent to $p$, see
\cite{gutev:05a}. Here, $\{S_n\}$ is \emph{convergent} to $p$ if every
neighbourhood of $p$ contains all but finitely many members of this
sequence; equivalently, if $\{S_n\}$ is $\tau_V$-convergent to
$\{p\}$. It was shown in \cite[Lemma 4.2]{gutev:05a} that if $p\in X$
is countably-approachable and $\sel[\mathscr{F}(X)]\neq\emptyset$,
then there exists a selection $f\in\sel[\mathscr{F}(X)]$ with
$f(X)=p$. Based on this, the following result was obtained in
\cite[Theorem 4.1]{gutev:05a}.

\begin{theorem}[\cite{gutev:05a}]
  \label{theorem-ext-sel-v28:1}
  Let $X$ be a space with $\sel[\mathscr{F}(X)]\neq \emptyset$. Then
  the set $\left\{f(X) : f\in \sel[\mathscr{F}(X)]\right\}$ is dense
  in $X$ if and only if the set of all countably-approachable points
  of X is dense in X.
\end{theorem}

Another interesting property is that each non-isolated
countably-approachable point of $X$ is a strong butterfly point
(equivalently, a cut point, or a tie-point of $X$). This follows
easily by taking a pairwise-disjoint sequence $\{S_n\}$ of nonempty
clopen sets such that $p\notin S_n$, $n < \omega$, and $\{S_n\}$ is
convergent to $p$. Then we can set $U=\bigcup_{k=0}^\infty S_{2k}$ and
$V=X\setminus\overline{U}$, see the proof of \cite[Corollary
3.2]{gutev-nogura:03a}.

\begin{theorem}
  \label{theorem-ext-sel-v26:3}
  If a space $X$ satisfies \eqref{eq:ext-sel-v26:2}, then each
  countably-approachable point is isolated.
\end{theorem}

\begin{proof}
  Assume that $X$ has a non-isolated countably-approachable point
  $p\in X$. Then by condition, there exists a pairwise-disjoint
  sequence $\{S_n\}$ of nonempty clopen sets such that $p\notin S_n$,
  $n < \omega$, and $\{S_n\}$ is convergent to $p$. Next, set
  \begin{equation}
    \label{eq:ext-sel-v26:4}
    \Omega_n=\left\{T\in \Sigma(X): T\in\langle S_0,\dots,
      S_n\rangle\right\}\quad\text{and}\quad
    \Omega=\bigcup_{n<\omega}\Omega_n.   
  \end{equation}
  Evidently, $\Omega_n$, $n<\omega$, is a pairwise-disjoint sequence
  of nonempty $\tau_V$-clopen sets in $\Sigma(X)$. We are going to
  show that $\Omega$ is also $\tau_V$-clopen in $\Sigma(X)$. To see
  this, let us observe that
  $Y=\overline{\bigcup_{n<\omega}S_n}=\{p\}\cup
  \left(\bigcup_{n<\omega}S_n\right)$ because the sequence
  $\{S_n\}\subset \mathscr{F}(X)$ is convergent to $p$. Next, take a
  set $S\in\Sigma(X)\setminus \Omega$. If $S\setminus Y\neq\emptyset$,
  then $\Theta=\{T\in \Sigma(X): T\setminus Y\neq\emptyset\}$ is
  $\tau_V$-open in $\Sigma(X)$ with $S\in \Theta$ and
  $\Theta\cap \Omega=\emptyset$. Otherwise, if $S\subset Y$, set
  $\mu=\{n<\omega: S\cap S_n\neq\emptyset\}$. If ${\mu=\emptyset}$,
  then $S=\{p\}$ and
  $\Theta=\left\{T\in \Sigma(X): T\cap S_0=\emptyset\right\}$ has the
  property that ${\Theta\cap \Omega=\emptyset}$. If $p\in S$ and
  $\mu\neq \emptyset$, let $m=\max \mu$ and, for convenience, set
  ${S_\mu=X\setminus \bigcup_{n=0}^{m+1}S_n}$ and
  $\mu^*=\mu\cup\{\mu\}$. Then
  $\Theta=\left\{T\in\Sigma(X): T\in\langle S_k: k\in
    \mu^*\rangle\right\}$ is a $\tau_V$-open set with $S\in \Theta$
  and $\Theta\cap \Omega=\emptyset$. Finally, if $p\notin S$, then
  $\mu$ is not an initial segment of $\omega$. Accordingly,
  $S\in \Theta$ and $\Theta\cap \Omega=\emptyset$ where
  $\Theta=\left\{T\in\Sigma(X): T\in\langle S_k: k\in
    \mu\rangle\right\}$.\smallskip

  We can now finish the proof as follows. Take a selection
  $h\in \sel[\Sigma(X)]$, and modify it on $\Omega$ to a selection
  $g\in\sel[\Sigma(X)]$ with the following property.  If
  $T\notin \Omega$, then set $g(T)=h(T)$. Otherwise, if $T\in \Omega$,
  then there is a unique $n(T)<\omega$ with $T\in \Omega_{n(T)}$, see
  \eqref{eq:ext-sel-v26:4}. In this case, set
  \begin{equation}
    \label{eq:ext-sel-v26:5}
    g(T)=
    \begin{cases}
      h(T\cap S_0) &\text{if $n(T)$ is even, and} \\
      h\left(T\cap S_{n(T)}\right) &\text{if $n(T)$ is odd.}
    \end{cases}
  \end{equation}
  Finally, let $f\in\sel[\mathscr{C}(X)]$ be an extension of $g$. To
  get a contradiction, take points $t_n\in S_n$, $n<\omega$, and set
  $T_n=\{t_0,\dots, t_n\}\in \Sigma(X)$ and
  $T=\overline{\bigcup_{n<\omega}T_n}\in \mathscr{C}(X)$. Then $T_n$,
  $n<\omega$, is $\tau_V$-convergent to $T$, so 
   $f(T)=\lim_{n\to\infty}f(T_n)=\lim_{n\to\infty}g(T_n)$.
  However, by \eqref{eq:ext-sel-v26:5}, this is impossible because
  $f(T)=\lim_{k\to\infty}g(T_{2k})=t_0$ and
  $f(T)=\lim_{k\to\infty}g(T_{2k+1})=p\neq t_0$. The proof is
  complete.
\end{proof}

For first countable spaces, Theorem \ref{theorem-ext-sel-v26:3}
implies the following consequence.

\begin{corollary}
  \label{corollary-ext-sel-v35:1}
  Let $X$ be a first countable space with a clopen $\pi$-base. If $X$
  satisfies \eqref{eq:ext-sel-v26:2}, then it is a discrete space and,
  in particular, $\mathscr{C}(X)=\Sigma(X)$.
\end{corollary}

\begin{proof}
  We follow the proof of \cite[Corollary 4.5]{gutev:05a}. Briefly,
  assume that $X$ has a non-isolated point $p\in X$, and take a
  decreasing open local base $\{V_n:n<\omega\}$ at this point. Since
  $X$ has a clopen $\pi$-base, there exists a nonempty clopen set
  $S_0\subset V_0\setminus\{p\}$. Set $n_0=0$ and take $n_1>n_0$ such
  that $S_0\cap V_{n_1}=\emptyset$. Just like before, we can take a
  nonempty clopen set $S_1\subset V_{n_1}\setminus\{p\}$ and to extend
  this construction by induction. Thus, there exists a subsequence
  $\{V_{n_k}:k<\omega\}$ of $\{V_n:n<\omega\}$ and nonempty clopen
  sets $S_k\subset V_{n_k}\setminus\{p\}$ such that
  $S_k\cap V_{n_{k+1}}=\emptyset$ for every $k<\omega$. Accordingly,
  the sequence $\{S_k:k<\omega\}$ is convergent to $p$, so $p$ is a
  countably-approachable point. However, this is impossible because
  $X$ satisfies \eqref{eq:ext-sel-v26:2} and by Theorem
  \ref{theorem-ext-sel-v26:3}, the point $p\in X$ must be isolated.
\end{proof}

As mentioned before, it was shown in \cite[Theorem
1.5]{gutev-nogura:00b} that $X$ is totally disconnected whenever the
set $\left\{f(X) : f\in \sel[\mathscr{F}(X)]\right\}$ is dense in
$X$. According to Theorem \ref{theorem-ext-sel-v28:1}, this is also
valid for a space $X$ which has a continuous selection for
$\mathscr{F}(X)$ and a dense set of countably-approachable
points. Furthermore, each space which has a dense set of
countably-approachable points also has a clopen $\pi$-base.  We now
have the following relaxed version of \cite[Corollary 2.3]{gutev:05a}.

\begin{proposition}
  \label{proposition-ext-sel-v35:1}
  Let $X$ be a space which has a continuous weak selection and a
  clopen $\pi$-base. Then $X$ is totally disconnected.
\end{proposition}

\begin{proof}
  Since $X$ has a continuous weak selection, by Theorem
  \ref{theorem-ext-sel-v16:1}, it suffices to show that each component
  of $X$ is a singleton. Assume that $X$ has a nontrivial
  component. Then the cut points of this component form an infinite
  open connected subset $U\subset X$. Moreover, by condition, any
  nonempty open subset of $U$ contains a nonempty clopen set. Hence,
  $U$ contains a nonempty clopen proper subset. Since $U$ is
  connected, this is impossible.
\end{proof}

This proposition and Corollary \ref{corollary-ext-sel-v35:1} suggest
the following natural question.

\begin{question}
  \label{question-ext-sel-v35:3}
  Let $X$ be a first countable totally disconnected space which
  satisfies \eqref{eq:ext-sel-v26:2}. Then, is it true that
  $X$ is discrete?
\end{question}

Regarding this question, let us remark that the space $X$ in Example
\ref{example-ext-sel-v35:1} is not discrete, but it has a dense set of
isolated points and $\Sigma(X)=\mathscr{C}(X)$. Hence, we also have
the following complementary question.

\begin{question}
  \label{question-ext-sel-v21:4}
  Let $X$ be a totally disconnected space which satisfies
  \eqref{eq:ext-sel-v26:2}. Then, is it true that
  $\Sigma(X)=\mathscr{C}(X)$? 
\end{question}

In the presence of the hyperspace selection-extension property in
\eqref{eq:ext-sel-v26:3}, we have the following further result which
is complementary to Corollary \ref{corollary-ext-sel-v35:1}.

\begin{corollary}
  \label{corollary-ext-sel-v35:2}
  Let $X$ be a first countable space which has a clopen $\pi$-base. If
  $X$ satisfies \eqref{eq:ext-sel-v26:3}, then it is a finite set.
\end{corollary}

\begin{proof}
  By Corollary \ref{corollary-ext-sel-v35:1}, $X$ is discrete because
  it satisfies \eqref{eq:ext-sel-v26:2}. Hence, $X$ is strongly cwH
  and satisfies \eqref{eq:ext-sel-v26:1}, so the property follows from
  Corollary \ref{corollary-ext-sel-v34:4}.
\end{proof}

We also have the following similar result in the setting of totally
disconnected countably compact spaces.

\begin{theorem}
  \label{theorem-ext-sel-v40:1}
  If $X$ is a countably compact totally disconnected space which
  satisfies \eqref{eq:ext-sel-v26:3}, then it is a finite set. 
\end{theorem}

\begin{proof}
  Since $X$ has a continuous weak selection and satisfies
  \eqref{eq:ext-sel-v26:1}, by Theorem \ref{theorem-ext-sel-v37:1} and
  Proposition \ref{proposition-ext-sel-v34:2}, $X$ is
  compact. Accordingly, $X$ is also zero-dimensional because it is
  totally disconnected. Suppose that $X$ contains a non-trivial
  sequence convergent to some point $p\in X$. Then as shown in
  \cite[Proposition 6.3]{MR3122363}, this point $p\in X$ must be
  countably-approachable. However, $p$ is a non-isolated point and by
  Theorem \ref{theorem-ext-sel-v26:3} this is impossible because $X$
  satisfies \eqref{eq:ext-sel-v26:2}. Thus, $X$ has no non-trivial
  convergent sequences which implies that each selection
  $f\in \sel[\mathscr{F}(X)]$ is a zero-selection for
  $\mathscr{F}(X)$. Indeed, take $f\in \sel[\mathscr{F}(X)]$ and
  assume that $f(T)$ is a non-isolated point of $T$ for some
  $T\in \mathscr{F}(X)$. Then by \cite[Proposition
  4.4]{garcia-ferreira-gutev-nogura-sanchis-tomita:99}, $T$ contains a
  non-trivial convergent sequence, but this is impossible. So, $f$ is
  a zero-selection and by \cite[Theorem 1]{fujii-nogura:99}, $X$ is an
  ordinal space. Since $X$ has no non-trivial convergent
  sequences, it must be a finite set being an ordinal space.  
\end{proof}

A scattered space is a special case of a space with a clopen
$\pi$-base. In fact, such a space has a clopen $\pi$-base of
singleton. According to the proof of Theorem
\ref{theorem-ext-sel-v40:1}, the space $X$ in this theorem is
scattered. In fact, as shown in that proof, each selection
$f\in \sel[\mathscr{F}(X)]$ is a zero-selection. This is also valid
for Corollaries \ref{corollary-ext-sel-v35:1} and
\ref{corollary-ext-sel-v35:2}, which were based on the fact that the
space has no non-trivial convergent sequences. Hence, it is
interesting to consider the following question.  

\begin{question}
  \label{question-ext-sel-v36:1}
  Let $X$ be a scattered space which satisfies
  \eqref{eq:ext-sel-v26:3}. Then, is it true that $X$ is a finite set
  (equivalently, compact)?
\end{question}

Another common element in Theorem \ref{theorem-ext-sel-v40:1} and
Corollaries \ref{corollary-ext-sel-v35:1} and
\ref{corollary-ext-sel-v35:2} is that each selection $f\in
\sel[\mathscr{F}(X)]$ is a zero-selection. This was explicitly shown
in Theorem \ref{theorem-ext-sel-v40:1}. The same argument is also
valid for Corollaries \ref{corollary-ext-sel-v35:1} and
\ref{corollary-ext-sel-v35:2} because the space in these consequences
has no non-trivial convergent sequences. Accordingly, we have the
following special case of Question \ref{question-ext-sel-v36:1}.

\begin{question}
  \label{question-ext-sel-v41:1}
  Let $X$ be space which satisfies \eqref{eq:ext-sel-v26:3}. If each
  selection $f\in \sel[\mathscr{F}(X)]$ is a zero-selection, then is
  it true that $X$ is a finite set (equivalently, compact)?
\end{question}

Regarding the possible difference between Questions
\ref{question-ext-sel-v36:1} and \ref{question-ext-sel-v41:1}, we also
have the following complementary question.

\begin{question}
  \label{question-ext-sel-v41:2}
  Let $X$ be a scattered space which satisfies
  \eqref{eq:ext-sel-v26:3}. Then, is it true that each selection
  $f\in \sel[\mathscr{F}(X)]$ is a zero-selection?
\end{question}

\section{Selections and Complete Graphs}
\label{sec:weak-extens-probl}

According to Theorem \ref{theorem-hsp-ver-5:7}, each weakly orderable
space $X$ has a continuous selection for $\mathscr{C}(X)$. In
particular, each weakly orderable space $X$ has a continuous selection
for $\mathscr{F}_n(X)$, for every $n\geq 2$. In view of the weak
orderability problem (Question \ref{question-ext-sel-v19:3}), this led
to another natural question posed in
\cite{gutev-nogura:02a,gutev-nogura:06b}.

\begin{question}[\cite{gutev-nogura:02a,gutev-nogura:06b}]
  \label{question-rpgt-29:5}
  Does there exist a space $X$ which has a continuous weak selection,
  but $\mathscr{F}_n(X)$ has no continuous selection for some $n > 2$?
\end{question}  

In contrast to Question \ref{question-ext-sel-v19:3}, the above
question is still open. In fact, it was shown by Hru{\v s}{\'a}k and
Mart{\'\i}nez-Ruiz \cite{hrusak-martinez:09} that the space in their
counterexample to Question \ref{question-ext-sel-v19:3} has a
continuous selection for $\mathscr{F}_n(X)$ for every $n\geq
2$.\medskip
 
The problem stated in Question \ref{question-rpgt-29:5} had an
interesting development. The first result obtained in this regard was
in \cite[Corollary 4.1]{gutev-nogura-garcia:04a}.

\begin{theorem}[\cite{gutev-nogura-garcia:04a}]
  \label{theorem-Selections-GRF:1}
  Let $X$ be a space which has a continuous selection for
  $[X]^3$. Then each continuous weak selection
  $g:\mathscr{F}_2(X)\to X$ can be extended to a continuous selection
  $f:\mathscr{F}_3(X)\to X$.
\end{theorem}

Subsequently, another special case was resolved in \cite[Theorem
4.1]{gutev-nogura:08a}, where the following further result was
obtained.

\begin{theorem}[\cite{gutev-nogura:08a}]
  \label{theorem-Selections-GRF:2}
  Let $X$ be a space which has a continuous selection for
  $\mathscr{F}_3(X)$. Then there exists a continuous selection for
  $\mathscr{F}_4 (X)$.
\end{theorem}

Theorems \ref{theorem-Selections-GRF:1} and
\ref{theorem-Selections-GRF:2} motivated the following two questions
in \cite[Questions 384 and 385]{gutev-nogura:06b}, see also
\cite[Question 1]{gutev-nogura-garcia:04a}.

\begin{question}[\cite{gutev-nogura-garcia:04a,gutev-nogura:06b}]
  \label{question-Selections-GRF:1}
  Let $X$ be a space such that for some $n\geq 3$, both
  $\mathscr{F}_n(X)$ and $[X]^{n+1}$ have continuous selections. Then
  is it true that $\mathscr{F}_{n+1} (X)$ also has a continuous
  selection?
\end{question}

\begin{question}[\cite{gutev-nogura:06b}]
  \label{question-ext-sel-v41:3}
  Let $X$ be a space which has a continuous selection for
  $\mathscr{F}_{2n+1}(X)$ for some $n\geq 2$. Then, is it true that
  $\mathscr{F}_{2n+2} (X)$ also has a continuous selection?
\end{question}

These questions were resolved in the affirmative in
\cite[Theorems 1.1 and 5.1]{gutev:08a}.

\begin{theorem}[\cite{gutev:08a}]
  \label{theorem-Selections-GRF:3}
    Let $n\geq 2$ and $X$ be a space such that both
  $\mathscr{F}_n(X)$ and $[X]^{n+1}$ have continuous
  selections. Then $\mathscr{F}_{n+1}(X)$ also has a continuous
  selection.  
\end{theorem}

\begin{theorem}[\cite{gutev:08a}]
  \label{theorem-Selections-GRF-v2:1}
  Let $n\geq 1$ and $X$ be a space which has a continuous selection
  for $\mathscr{F}_{2n+1}(X)$. Then $\mathscr{F}_{2n+2}(X)$ also has a
  continuous selection.
\end{theorem}

On the other hand, Theorem \ref{theorem-Selections-GRF:1} claims a
stronger extension property which suggests the following natural
question.

\begin{question}
  \label{question-Selections-GRF-v4:1}
  Let $X$ be a space which has a continuous selection for
  $\mathscr{F}_{n+1}(X)$ for some $n> 2$. Then, is it true that each
  continuous selection $g:\mathscr{F}_n(X)\to X$ can be extended to a
  continuous selection $f:\mathscr{F}_{n+1}(X)\to X$?  In particular,
  is it true that each continuous selection $g:\mathscr{F}_3(X)\to X$
  can be extended to a continuous selection $f:\mathscr{F}_4(X)\to X$?
\end{question}

Here is also a question is about the role of the hyperspace
$[X]^n$ in Theorem \ref{theorem-Selections-GRF:3}.

\begin{question}
  \label{question-ext-sel-v41:5}
  Let $X$ be a space with a continuous weak selection. If $[X]^{n}$
  has continuous selection for some $n>2$, then is it true that
  $\mathscr{F}_{n}(X)$ also has a continuous selection?
\end{question}

Each separable metrizable space with a continuous weak selection is
weakly orderable, see \cite[Theorem 1.1]{gutev:07a} and \cite[Theorem
4.1]{MR3430989}. Similarly, each strongly zero-dimensional metrizable
space is (weakly) orderable as well, see Theorem
\ref{theorem-hsp-ver-10:9}. Thus, complementary to Question
\ref{question-ext-sel-v19:2}, we also have the following question.

\begin{question}
  \label{question-Selections-GRF-v4:3}
  Let $X$ be a (zero-dimensional) metrizable space with a continuous
  weak selection. Then, is it true that $[X]^3$ also has a continuous
  selection?
\end{question}

Another related hyperspace selection-extension property was obtained
in \cite[Theorem 5.3]{gutev-nogura-garcia:04a}, compare with Theorem
\ref{theorem-Selections-GRF:1}.

\begin{theorem}[\cite{gutev-nogura-garcia:04a}]
  \label{theorem-ext-sel-v41:1}
  Let $X$ be a space and $p\in X$ be such that
  $\mathscr{F}_3(X\setminus\{p\})$ has a continuous selection. Then
  each continuous weak selection $g:\mathscr{F}_2(X)\to X$ can be
  extended to a continuous selection $f:\mathscr{F}_3(X)\to X$.
\end{theorem}

Combining Theorems \ref{theorem-Selections-GRF:1} and
\ref{theorem-ext-sel-v41:1}, we have the following immediate
consequence.

\begin{corollary}
  \label{corollary-ext-sel-v41:1}
  Let $X$ be a space and $p\in X$ be such that $[X\setminus\{p\}]^3$
  has a continuous selection. Then each continuous weak selection
  $g:\mathscr{F}_2(X)\to X$ can be extended to a continuous selection
  $f:\mathscr{F}_3(X)\to X$.
\end{corollary}

In \cite[Corollary 5.4]{gutev-nogura-garcia:04a}, Theorem
\ref{theorem-ext-sel-v41:1} was used to show that for a space $X$ with
only one non-isolated point $p\in X$, each continuous weak selection
$g: \mathscr{F}_2(X)\to X$ can be extended to a continuous selection
$f:\mathscr{F}_3(X)\to X$. This brings the following question for
scattered spaces which is complementary to Question
\ref{question-Selections-GRF-v4:3}. 

\begin{question}
  \label{question-ext-sel-v42:1}
  Let $X$ be a (zero-dimensional) scattered space with a continuous
  weak selection. Then, is it true that $[X]^3$ also has a continuous
  selection?
\end{question}

Regarding scattered spaces, the following special case of Question
\ref{question-ext-sel-v42:1} was resolved in \cite[Theorem
3.2]{Jiang2007}.

\begin{theorem}[\cite{Jiang2007}]
  \label{theorem-ext-sel-v41:2}
  If a scattered hereditarily paracompact space $X$ has a continuous
  weak selection, then it also has a continuous selection for
  $\Sigma(X)$.
\end{theorem}

The following similar result was obtained in \cite[Corollary
3.7]{hrusak-martinez:09}.

\begin{theorem}[\cite{hrusak-martinez:09}]
  \label{theorem-ext-sel-v29:1}
  If a separable space $X$ has a continuous weak selection, then it
  also has a continuous selection for $\Sigma(X)$.
\end{theorem}

Theorem \ref{theorem-ext-sel-v29:1} was based on the technique of
decisive partitions developed in \cite{gutev:08a}, see \cite[Theorem
3.3]{gutev:08a}, which brings the following natural question.

\begin{question}
  \label{question-ext-sel-v29:1}
  Let $X$ be a space which has a continuous selection for $[X]^n$ for
  every $n\geq 2$. Then, is it true that $\Sigma(X)$ also has a
  continuous selection?
\end{question}

It was shown in \cite[Proposition 3.8]{hrusak-martinez:09} that there
exists a separable space $X$ which is not weakly orderable, but
$\sel[\mathscr{C}(X)]\neq \emptyset$. On this basis, the following
question was stated in \cite[Question 4.3]{hrusak-martinez:09}.

\begin{question}[\cite{hrusak-martinez:09}]
  \label{question-ext-sel-v41:4}
  Does every (separable) space $X$ with a continuous weak selection
  also admit a continuous selection for $\mathscr{C}(X)$?
\end{question}

We now briefly discuss the proof of these theorems which was based on
some combinatorial arguments. In fact, the proof of Theorem
\ref{theorem-Selections-GRF:2} in \cite{gutev-nogura:08a} was based on
a natural interpretation of weak selections as flows on complete
graphs. A \emph{graph} is a pair $G=(V,E)$ of sets satisfying
$E\subset [V]^2$. The elements of $V$ are the \emph{vertices} (or
\emph{nodes}, or \emph{points}) of the graph $G$, while the elements
of $E$ are its \emph{edges} (or \emph{lines}). The vertex set of a
graph $G$ is referred to as $\mathbf{V}(G)$, its edge set as
$\mathbf{E}(G)$.  In what follows, we will consider only finite
graphs. Moreover, for convenience, we will identify the vertex set
$\mathbf{V}(G)$ of a graph $G$ with the graph itself, namely
$\mathbf{V}(G)=G$. Thus, the edges $\mathbf{E}(G)$ of $G$ are simply a
subset of $[G]^2$.\medskip

Two vertices $x,y\in G$ of a graph $G$ are \emph{adjacent}, or
\emph{neighbours}, if the set $\{x,y\}$ is an edge of $G$, i.e.\
$\{x,y\}\in \mathbf{E}(G)$. If all vertices of $G$ are pairwise
adjacent, then $G$ is \emph{complete}. A complete graph of $n$
vertices is denoted by $K_n$, i.e.\ $K_n=(K_n,[K_n]^2)$.\medskip

A graph $G$ can be considered as a \emph{network} when its edges carry
some kind of a flow. Namely, let
\begin{equation}
  \label{eq:ext-sel-v42:2}
  \vec{\mathbf{E}}(G)=\big\{(x,y),(y,x): \{x,y\}\in \mathbf{E}(G)\big\}.
\end{equation}
A function $\xi:\vec{\mathbf{E}}(G)\to \Z$ in the integers $\Z$ is
called a \emph{flow} on the graph $G$ if $\xi(x,y)=-\xi(y,x)$,
$(x,y)\in \vec{\mathbf{E}}(G)$. Thus, $\xi(x,y)$ expresses that a flow
of $\xi(x,y)$-units passes through the edge $e=\{x,y\}$ from $x$ to
$y$, while $\xi(y,x)=-\xi(x,y)$ are the units of flow that passes
through $e$ the other way, from $y$ to $x$. Finally, for a flow $\xi$
on $G$ and $a\in G$, we define the \emph{node} $\mathbf{N}(a)$ of this
vertex and the \emph{total flow} $\varphi_\xi(a)$ through this node by
\begin{equation}
  \label{eq:Selections-GRF-v2:1}
  \begin{cases}
  \mathbf{N}(a)=\big\{b\in G\setminus\{a\}: \{a,b\}\in
\mathbf{E}(G)\big\},\quad \text{and}\\
  \varphi_\xi(a)=\sum_{b\in \mathbf{N}(a)} \xi(a,b).
\end{cases}
\end{equation}
Since $\xi(x,y)+\xi(y,x)=0$, we have the following simple observation.

\begin{proposition}[\cite{gutev-nogura:08a}]
  \label{pr:flow-general}
  If $\xi:\vec{\mathbf{E}}(G) \to \Z$ is a flow on $G$, then $\sum_{a\in
    G}\varphi_\xi(a)=0$.
\end{proposition}

In this setting, we say $\xi:\vec{\mathbf{E}}(G)\to \Z$ is a
\emph{selection flow} on a graph $G$ if it takes values in the
$0$-dimensional sphere $\s^0=\{-1,1\}\subset \Z$, i.e.\ when
$\xi:\vec{\mathbf{E}}(G)\to \s^0$. This is motivated by the following
example in \cite[Example 3.2]{gutev-nogura:08a}.

\begin{example}[\cite{gutev-nogura:08a}]
  \label{ex:selection-flow}
  For a weak selection $\sigma$ on $X$ and $S\in [X]^n$, where
  ${n\geq 2}$, consider the complete graph $K_n=S$, and the selection
  flow ${\xi_S:\vec{\mathbf{E}}(S)\to \s^0}$ defined for $x,y \in S$
  by $\xi_S(x,y)=1$ if $x<_\sigma y$ and $\xi_S(x,y)=-1$ if
  $y<_\sigma x$. Then such flows are ``locally constant'' on $[X]^n$
  whenever $\sigma$ is continuous. Namely, continuity of $\sigma$
  implies the existence of a family $\mathscr{V}=\{V_x: x\in S\}$ of
  open subsets of $X$ such that $x\in V_x$, $x\in S$, and
  $V_x<_\sigma V_y$ precisely when $x<_\sigma y$, see \cite[Corollary
  2.2]{gutev-nogura:08a}.  Evidently, in this case, $\mathscr{V}$ is a
  pairwise disjoint family such that $|\mathscr{V}|=|S|=n$ and
  $V<_\sigma W$ or $W<_\sigma V$ for every two distinct members
  $V,W\in \mathscr{V}$. Such a basic $\tau_V$-open set
  $\langle\mathscr{V}\rangle$ was called \emph{$\sigma$-decisive} in
  \cite{gutev-nogura:08a}.  Thus, continuity of $\sigma$ means that
  for every $T\in \langle\mathscr{V}\rangle\cap [X]^n$, the
  corresponding selection flow $\xi_T:\vec{\mathbf{E}}(T)\to \s^0$ is
  actually identical to $\xi_S$ because $\langle\mathscr{V}\rangle$ is
  $\sigma$-decisive and the flow is defined only in terms of the
  relation $\leq_\sigma$. Accordingly, each continuous weak selection
  $\sigma$ defines a discrete partition of $[X]^n$, and each element
  of this partition corresponds to a fixed (selection) flow
  $\xi:\vec{\mathbf{E}}(K_n)\to \s^0$ on $K_n$. Evidently, there are
  only finitely many such flows on $\vec{\mathbf{E}}(K_n)$.\qed
\end{example}

The converse is also true and if $\xi:\vec{\mathbf{E}}(K_n)\to \s^0$
is a selection flow, then it defines a weak selection
$\sigma_\xi:[K_n]^2\to K_n$ by letting $x<_{\sigma_\xi} y$ precisely
when $\xi(x,y)=1$. In view of this interpretation, Theorems
\ref{theorem-Selections-GRF:2} and \ref{theorem-Selections-GRF-v2:1}
are based on the following ``partition'' problem on complete graphs.

\begin{question}
  \label{question-Selections-GRF-v2:1}
  Let $m>n\geq 2$ and $\sigma:[K_{m}]^n\to K_{m}$ be a
  selection. Does there exist a partition $\{A,B\}$ of $K_{m}$ which
  depends only on $\sigma$?
\end{question}

To understand properly Question \ref{question-Selections-GRF-v2:1},
let us consider the following example. For the complete graph
$K_3=\{x,y,z\}$ consisting of $3$ points, define a selection
$\sigma:[K_3]^2\to K_3$ by $\sigma(\{x,y\})=x$,
$\sigma(\{y,z\})=y$ and $\sigma(\{z,x\})=z$. The relation
$\leq_\sigma$ generated by this weak selection has the property that
\[
  \dots <_\sigma x<_\sigma y<_\sigma z<_\sigma x<_\sigma \cdots
\]
In other words, $\sigma$ cannot distinguish between the points of
$K_3$, and for this selection the answer to Question
\ref{question-Selections-GRF-v2:1} is ``No''. \medskip

In contrast, Question \ref{question-Selections-GRF-v2:1} doesn't make
sense for Theorem \ref{theorem-Selections-GRF:3}. A selection $\sigma$
for $[K_{n+1}]^{n+1}$ defines in a trivial way a partition
$A=\{\sigma(K_{n+1})\}$ and $B=K_{n+1}\setminus A$ of
$K_{n+1}$. Hence, the problem in Theorem
\ref{theorem-Selections-GRF:3} is to modify the given selections on
$\mathscr{F}_{n}(X)$ and $[X]^{n+1}$ by transforming them into a
continuous selection for $\mathscr{F}_{n+1}(X)$. The situation with
Theorems \ref{theorem-Selections-GRF:2} and
\ref{theorem-Selections-GRF-v2:1} is different, they are based on a
solution of Question \ref{question-Selections-GRF-v2:1} which gives a
selection on $[K_{n+1}]^{n+1}$, i.e.\ on $[X]^{n+1}$. In other words,
these theorems are presenting a reduction to Theorem
\ref{theorem-Selections-GRF:3}. This is briefly discussed below.

\begin{proposition}[\cite{gutev-nogura:08a}]
  \label{proposition-Selections-GRF-v2:1}
  Let $\xi :\vec{\mathbf{E}}(K_4)\to\s^0$ be a selection flow on
  $K_4$, and $\varphi_\xi: K_4\to \Z$ be defined as in
  \eqref{eq:Selections-GRF-v2:1}. Then one of the following
  holds.
  \begin{enumerate}[label=\upshape{(\roman*)}]
  \item There exists a point $a\in K_4$ with $|\varphi_\xi (a)| = 3$.
  \item  $|\varphi_\xi (a)| = 1$ for every $a\in K_4$.
  \end{enumerate}
\end{proposition}

Since $\xi :\vec{\mathbf{E}}(K_4)\to\s^0$ is a flow, it follows easily
from Proposition \ref{proposition-Selections-GRF-v2:1} that $K_4$ has
a partition into subsets $A$ and $B$. In fact, by
\eqref{eq:Selections-GRF-v2:1} and Proposition \ref{pr:flow-general},
we can take $A=\{a\in K_4: \varphi_\xi(a)>0\}$ and
$B=\{b\in K_4: \varphi_\xi(b)<0\}$. For an arbitrary $n\geq1$, this
property was extended to $K_{2n+2}$ in \cite{gutev:08a}.

\begin{proposition}[\cite{gutev:08a}]
  \label{proposition-partition-even}
  For a selection flow $\xi:\vec{\mathbf{E}}(K_{2n+2})\to \s^0$ on
  $K_{2n+2}$, $n\geq 1$, let\,\
  $A=\{a\in K_{2n+2}: \varphi_\xi(a)>0\}$\,\ and\,\
  $B=\{b\in K_{2n+2}: \varphi_\xi(b)<0\}$.  Then $A\cap B=\emptyset$
  and $A\neq\emptyset\neq B$.
\end{proposition}

This simple property was used in \cite{gutev:08a} to show that
$[X]^{2n+2}$ has a continuous selection provided so does
$\mathscr{F}_{2n+1}(X)$. Accordingly, Theorem
\ref{theorem-Selections-GRF-v2:1} follows from Theorem
\ref{theorem-Selections-GRF:3}. In fact, Theorems
\ref{theorem-Selections-GRF:3} and \ref{theorem-Selections-GRF-v2:1}
can be combined in the following general statement, it follows easily
by these theorems.

\begin{theorem}
  \label{theorem-ext-sel-v42:1}
  For a space $X$ with a continuous weak selection, the following are
  equivalent\textup{:}
  \begin{enumerate}
  \item $\mathscr{F}_n(X)$ has a continuous selection for every
    $n\geq 3$.
  \item $[X]^n$ has a continuous selection for every
    $n\geq 3$.
  \item $[X]^n$ has a continuous selection for every odd $n\geq 3$.
  \end{enumerate}
\end{theorem}

This result was recently generalised in \cite[Theorem
5]{CruzChapital2022} by replacing odd numbers with prime numbers.

\begin{theorem}[\cite{CruzChapital2022}]
  \label{theorem-Selections-GRF-v2:2}
  For a space $X$, the following are equivalent\textup{:}
  \begin{enumerate}
  \item $\mathscr{F}_n(X)$ has a continuous selection for every
    $n\geq 2$.
  \item $[X]^n$ has a continuous selection for every $n\geq 2$.
  \item $[X]^n$ has a continuous selection for every prime $n\geq 2$.
  \end{enumerate}
\end{theorem}

Theorem \ref{theorem-Selections-GRF-v2:2} was based on another
solution of Question \ref{question-Selections-GRF-v2:1} in
\cite{CruzChapital2022}. Briefly, for $m> p\geq 2$ and a selection
$\sigma:[K_m]^p\to K_m$, define a function
$\varkappa_\sigma:K_m\to \omega$ by
\begin{equation}
  \label{eq:Selections-GRF-v2:2}
  \varkappa_\sigma(x)=|\sigma^{-1}(x)|\quad \text{for every $x\in K_m$.}
\end{equation}
In other words, to each $x\in K_m$ the function $\varkappa_\sigma$
associates the number of the elements of the preimage
$\sigma^{-1}(x)$. The following interesting result was obtained in
\cite[Lemma 3]{CruzChapital2022}.

\begin{proposition}[\cite{CruzChapital2022}]
  \label{proposition-Selections-GRF-v2:2}
  If $p$ is a prime number, $m> p$ and $\sigma:[K_m]^p\to K_m$ is a
  selection such that $\varkappa_\sigma:K_m\to \omega$ is a constant
  function, then $p$ doesn't divide~$m$.  
\end{proposition}

The following consequence of Proposition
\ref{proposition-Selections-GRF-v2:2} illustrates the corresponding
solution of Question \ref{question-Selections-GRF-v2:1} related to the
proof of Theorem \ref{theorem-Selections-GRF-v2:2}.

\begin{corollary}[\cite{CruzChapital2022}]
  \label{corollary-Selections-GRF-v2:2}
  Let $p$ be a prime number and $\sigma:[K_m]^p\to K_m$ be a selection
  for some $m> p$. If $p$ divides $m$, then
  $\varkappa_\sigma(x)\neq\varkappa_\sigma(y)$ for some points $x,y\in
  K_m$. 
\end{corollary}

Intuitively, in the setting of Corollary
\ref{corollary-Selections-GRF-v2:2}, let
$\kappa=\min\{\varkappa_\sigma(x): x\in K_m\}$. Then by this
corollary, $A=\left\{x\in K_m: \varkappa_\sigma(x)=\kappa\right\}$ is
a nonempty subset of $K_m$ with $A\neq K_m$. Hence, it can be used to
define the required partition of $K_m$.\medskip

Answering a question of \cite{hrusak-martinez:09}, it was shown in
\cite[Example 2.5]{MR3705772} that for every $n\geq 2$ there exists a
connected second countable space $X$ such that $[X]^{n+1}$ has a
continuous selection, but $[X]^k$ has no continuous selection for
every $2\leq k\leq n$. It was further shown in \cite[Corollary
5.5]{MR3705772} that if $X$ is a connected space with a continuous
selection for $[X]^n$ for some $n\geq 2$, then $[X]^{n+1}$ has at
least two continuous selections. Evidently, in this case, $[X]^k$
has a continuous selection for every $k\geq n$. Based on this, the
following question was posed in \cite[Problem 3]{CruzChapital2022}.

\begin{question}[\cite{CruzChapital2022}]
  \label{question-ext-sel-v42:2}
  For a space $X$ and $m>n\geq 2$, does the existence of a continuous
  selection for $\mathscr{F}_m(X)\setminus \mathscr{F}_n(X)$ imply
  that $[X]^k$ also has a continuous selection for some $k>m$?
\end{question}

Here is also a related question.

\begin{question}
  \label{question-ext-sel-v42:3}
  Does there exist an integer $n\geq 2$ such that if $X$ is a space
  which has a continuous selection for $\mathscr{F}_n(X)$, then
  $\mathscr{F}_{k}(X)$ also has a continuous selection for every
  $k\geq n$?
\end{question}

\section{Appendix: Selections and Tournaments}
\label{sec:append-tourn-select}

In this Appendix, we will briefly summarise several captivating
interpretations of weak selections in terms of tournaments on complete
graphs which are complementary to the results of the previous
section. They illustrate a natural relationship with the hyperspace
selection problem --- a problem with a fascinating history and
appealing applications.

\subsection*{Tournaments}

A \emph{directed graph}, or \emph{digraph} for short, consists of a
finite set $V$ of points (its vertices) together with a subset of
$V\times V$, whose elements are called \emph{lines}. Each line
$(u, v) =uv$ is directed and goes from its first point $u$ to a
different second point $v$, so that a digraph is
\emph{irreflexive}. In this case, we say that $u$ is \emph{adjacent
  to} $v$ and $v$ is \emph{adjacent from} $u$.  A digraph is
\emph{asymmetric} if whenever a line $uv$ is in it, then $vu$ is not
in it. A digraph is \emph{transitive} if for every three distinct
points $u, v$ and $w$, the existence of lines $uv$ and $vw$ implies
that $uw$ is also a line in it. In a complete digraph, for any two
vertices $u$ and $v$, there exists one of the lines $uv$ or $vu$. A
\emph{tournament} is a complete asymmetric digraph. \medskip

To each digraph we may associate the ``reverted'' direct graph
commonly called the \emph{converse}. It is simply defined by reverting
each line, in other words changing adjacent ``to'' to adjacent
``from''. A valuable principle in the theory of directed graphs is
that each theorem about digraphs has a corresponding theorem which is
obtained by replacing every concept by its converse. So, in some
sense, studying digraphs we may identify each such graph with its
converse.

\subsection*{Weak Selections}

Each digraph with a vertex set $V$ can be identified with a binary
relation ``$<$'' on $V$ such that $uv$ is line precisely when $v<u$.
This relation has all properties of a partial order on $V$ except of
being transitive. In fact, in these terms, a digraph is transitive
precisely when its vertex set is partially ordered in the sense of the
above relation. Furthermore, for a tournament $T$ on $V$, the
corresponding relation $<_T$ on $V$ is a ``strict'' \emph{selection
  relation} on $V$ in the sense of \cite{gutev-nogura:03b}. Namely,
the tournament $T$ on $V$ is identical to a weak selection
$\sigma:[V]^2\to V$ for which $u<_T v$ precisely when $u<_\sigma v$,
i.e.\ when $\sigma(\{u,v\})=u$, see page
\pageref{page-strict-order}. Evidently, the converse is also true, and
each weak selection for $V$ defines a tournament on $V$. Thus,
tournaments and weak selections are actually identical in the setting
of finite sets. In fact, this interpretation was implicitly used in
\cite{hrusak-martinez:09} and \cite{MR2944781}, where for a weak
selection $\sigma$ on $V$, the authors used $u\to_\sigma v$ to express
that $\sigma(\{u,v\})=v$. In what follows, a tournament on $V$ will be
identified with the corresponding weak selection $\sigma:[V]^2\to
V$. In this setting, $uv$ is a line for some $\{u,v\}\in [V]^2$ if and
only if
\begin{equation}
  \label{eq:ext-sel-v43:1}
  \sigma(\{u,v\})=v\quad\iff\quad v<_\sigma u\quad\iff\quad
  u\to_\sigma v.
\end{equation}
We may extend this idea to digraphs, but here we will be mainly
interested in tournaments.

\subsection*{Pecking Order}

It was through the careful observation of a backyard chicken flock
that one of the most important principles of social biology was
uncovered --- one that applies equally well to humans. For over 4,000
years, since the Red Jungle Fowl was first domesticated in Southern
Asia, farmers had noticed that a flock of hens was a very orderly
group.  At feeding time, the dominant birds in the flock would eat
first, picking out the best morsels, then came the less-dominant
birds, and finally the least dominant who got whatever was
left. Farmers knew that if anything happened to disrupt this order,
there would be a brief period of discord as birds fought with each
other to re-establish dominance, then peace would reign once
again. The dominant chicken asserts its dominance by pecking the other
chicken on the head and neck, whence the phrase ``\emph{pecking
  order}''. However, it rarely happens that this pecking order is
linear. That is, it is rare that there is a first chicken who pecks
all the others, a second chicken who pecks all but the first, and so
on. The term ``\emph{pecking order}'' was coined by Thorleif
Schjelderup-Ebbe in 1921, when he studied dominance hierarchy in
chickens and other birds. Not surprisingly, this ``pecking order'' is
nothing else but the relation $<_\sigma$ generated by a weak selection
$\sigma$ on a finite flock of chickens $V$, i.e.\ a tournament
$\sigma:[V]^2\to V$.

\begin{question}
  \label{question-Selections-GRF-v3:1}
  Given that the pecking order may not be linear, is there still a
  reasonable way to designate a most dominant chicken, i.e.\ a
  \emph{king}?
\end{question}

In 1951, the question was resolved by H. G. Landau
\cite{MR0041412,MR54932,MR54933} (see also \cite{MR567954}) who showed
that any finite flock of chickens has a most dominant one, namely a
\emph{king}. Landau's mathematical model was based on Graph Theory and
became known as ``The King Chicken Theorem''. To state his result, we
briefly review some concepts regarding tournaments or, equivalently,
weak selections. \medskip

For a tournament $\sigma:[V]^2\to V$ on a finite set $V$, the
\emph{outdegree} of $v\in V$, denoted by $\od_\sigma(v)$, is the
number of points adjacent from $v$; its \emph{indegree},
$\id_\sigma(v)$, is the number of points adjacent to $v$. In other
words, according to \eqref{eq:ext-sel-v43:1}, 
\begin{equation}
  \label{eq:ext-sel-v43:2}
  \begin{cases}
    \od_\sigma(v)=\big|\{u\in V: v\to_\sigma u\}\big|=\big|\{u\in V:
    u<_\sigma v\}\big|,\quad \text{and}\\
    \id_\sigma(v)\hspace{3pt}=\big|\{u\in V: u\to_\sigma v\}\big|=\big|\{u\in V:
    v<_\sigma u\}\big|= |\sigma^{-1}(v)|.
  \end{cases}
\end{equation}

A \emph{transmitter} is a point with positive outdegree and zero
indegree; a \emph{receiver} has positive indegree and zero
outdegree. In terms of the pecking order $<_\sigma$, a point $v\in V$
is a transmitter precisely when it is the $\leq_\sigma$-maximal
element of $V$:
\begin{equation}
  \label{eq:ext-sel-v43:3}
  v\to_\sigma u\quad\text{for every $u\in V\setminus\{v\}$.}
\end{equation}
Similarly, $v\in V$ is a receiver when it is the $\leq_\sigma$-minimal
element of $V$. Here, let us explicitly remark that $<_\sigma$ is not
necessarily a linear order on $V$, hence it may fail to be
transitive. In \cite{MR567954}, a point $v\in V$ was called a
\emph{$\sigma$-emperor} if it is the $\leq_\sigma$-maximal element of
$V$, i.e.\ when $v$ is a transmitter for the tournament $\sigma$ on
$V$. Thus, $V$ can have at most one $\sigma$-emperor (equivalently, a
transmitter). The dual concept was also defined in \cite{MR567954},
and $v\in V$ was called a \emph{$\sigma$-slave} if it is the
$\leq_\sigma$-minimal element of $V$, i.e.\ a receiver. In other
words, a slave $v\in V$ is ``pecked'' by all other elements (chickens)
of $V$. An element of $V$ which is neither a $\sigma$-emperor nor a
$\sigma$-slave, was called a \emph{$\sigma$-serf} \cite{MR567954}. The
concept of a $\sigma$-emperor is very restrictive, somehow
representing an extreme case. A less restrictive concept is a
$\sigma$-king. An element $v\in V$ is a \emph{$\sigma$-king}
\cite{MR567954} if for every other $u\in V$ there exists $w\in V$ such
that $u \leq_\sigma w\leq_\sigma v$, namely when 
\begin{equation}
  \label{eq:ext-sel-v43:4}
  v\to_\sigma u\qquad\text{or}\qquad v\to_\sigma w\to_\sigma
  u,\ \ \text{for some $w\in V$.}
\end{equation}
The following theorem was proved by Landau \cite{MR54933}, for the
second part see \cite{MR567954}.  

\begin{theorem}[\cite{MR54933,MR567954}]
  \label{theorem-Selections-GRF-v3:1}
  Every tournament $\sigma:[V]^2\to V$ on a finite set $V$ has a
  $\sigma$-king. Moreover, $V$ has exactly one $\sigma$-king if and
  only if that $\sigma$-king is a $\sigma$-emperor.
\end{theorem}

Landau's result was extended to infinite tournaments on topological
spaces by Nagao and Shakhmatov. In \cite{MR2944781}, they called a
topological space $X$ to be a \emph{king space} if $X$ has a
continuous weak selection, and every continuous weak selection
$\sigma$ for $X$ has a \emph{$\sigma$-king}, i.e.\ a point $q\in X$
such that for every $x\in X$ there exists $y\in X$ with
$x\leq_\sigma y\leq_\sigma q$. Next, they showed that every compact
space with a continuous weak selection is a king space \cite[Theorem
2.3]{MR2944781}. In the inverse direction, Nagao and Shakhmatov showed
that each king space is compact whenever it is linearly ordered
(\cite[Corollary 3.3]{MR2944781}); or pseudocompact, or
zero-dimensional, or locally connected (\cite[Theorem
3.5]{MR2944781}). Subsequently, answering a question of
\cite{MR2944781}, it was shown in \cite[Theorem 4.1]{MR3640030} that
each locally pseudocompact king space is also compact. A somewhat
weaker concept of a king space, called a \emph{quasi-king space}, was
introduced in \cite{Gutev2018b}. It was motivated by some related
results about coarser compact topologies, see \cite[Theorem
1.1]{Gutev2018b}.

\subsection*{``Bums'' and ``Deadbeats''}

The directional dual of the next theorem was stated by Silverman
\cite{silverman1962problem} in the following picturesque terminology:
Consider a club in which among any two members, one is a creditor of
the other. A ``\emph{bum}'' is defined as a member who is in debt to
everyone else. A ``\emph{deadbeat}'' is not a ``bum'', but for each
member he does not owe, he owes someone who owes this member. Then if
the club has no bums, it has at least three deadbeats.\medskip

\begin{theorem}
  \label{theorem-Selections-GRF-v4:1}
  If a tournament $\sigma$ on $V$ has no $\sigma$-emperor, then it contains
  at least three points each of which can reach every point in at most
  two steps.
\end{theorem}

The idea of Theorem \ref{theorem-Selections-GRF-v4:1} is simple. By
Theorem \ref{theorem-Selections-GRF-v3:1}, $V$ has a $\sigma$-king
$v\in V$. Since $v$ is not a $\sigma$-emperor, by
\eqref{eq:ext-sel-v43:3}, $u\to_\sigma v$ for some $u\in V$. However,
$v$ is a $\sigma$-king, and it follows from \eqref{eq:ext-sel-v43:4}
that $v\to_\sigma w\to_\sigma u$ for some $w\in V$. Evidently, this
defines a cycle $v\to_\sigma w\to_\sigma u\to_\sigma v$ and each one
of these points is a $\sigma$-king.

\subsection*{Round-Robin Tournaments}

A \emph{round-robin} tournament (or \emph{all-play-all} tournament) is
a competition in which each contestant meets all other contestants in
turn. It follows from \eqref{eq:ext-sel-v43:2} that if $v\in V$ is any
point in a tournament $\sigma$ on a set $V$ with $n\geq2$ points, then
$\od_\sigma(v)+\id_\sigma(v)=n-1$. In the tournament of a round-robin
competition, the outdegree of a point is the number of victories won
by that player. For this reason, we shall call the outdegree of a
point $v$ of a tournament its \emph{score}, denoted by
$s_\sigma(v)$. In case the elements of $V$ are ordered as
$\{v_1,v_2,\dots, v_n\}$, we will simply write that
$s_i=s_\sigma(v_i)$. The \emph{score sequence} of a tournament
$\sigma$ is the ordered sequence of integers
$\mathbf{S}_\sigma=(s_1, s_2,\dots, s_n)$ corresponding to the ordered
set of vertices $V=\{v_1,v_2,\dots, v_n\}$. It will be convenient to
order the points $v_i$ in such a way that
$s_1\leq s_2\leq \dots \leq s_n$. The following theorem by Landau
\cite{MR54933} gives a necessary and sufficient condition for a
sequence of non-negative integers to be the scores of some
tournament. It is based on the fact that the total number of lines in
a tournament $\sigma$ on $V$ is $\binom{n}2=\frac{n(n-1)}2$.

\begin{theorem}[\cite{MR54933}]
  \label{theorem-Selections-GRF-v3:2}
  A sequence of non-negative integers $s_1\leq s_2\leq \dots \leq s_n$
  is a score sequence for a tournament $\sigma$ on a set
  $V=\{v_1,v_2,\dots, v_n\}$ if and only if
  \begin{equation}
    \label{eq:Selections-GRF-v3:1}
    \sum_{i=1}^n s_i= \frac{n(n-1)}2\qquad \text{and}\qquad \sum_{i=1}^k
    s_i\geq \frac{k(k-1)}2,\ 1\leq k<n. 
  \end{equation}
\end{theorem}

To illustrate the above theorem, consider a basketball league
consisting of ten teams in which each team plays every other team
once, see \cite{Harary1966}. Since no game can end in a tie, the
digraph of the outcomes of all games at the end of the season is a
tournament. What are the possible distributions of the number of
victories among the teams?  Clearly, each distribution must satisfy
condition \eqref{eq:Selections-GRF-v3:1} of Theorem
\ref{theorem-Selections-GRF-v3:2}. This fact provides information
concerning certain questions that may be asked about the final
standings of the team. For example, what is the largest number of
teams that can have a winning season? The answer is nine, since the
sequence of integers $(0, 5, 5, 5, 5, 5, 5, 5, 5, 5)$ satisfies the
condition of Theorem \ref{theorem-Selections-GRF-v3:2} and no sequence
containing ten integers, all greater than 4, does. Can the season end
in a complete tie? Clearly not, since the average of the scores is not
an integer.\medskip

Here is another example. For the complete graph $K_5$ of five
vertices, according to Theorem \ref{theorem-Selections-GRF-v3:2},
there exists a tournament $\sigma:[K_5]^2\to K_5$ with a score
sequence $(2,2,2,2,2)$. In other words, $\sigma$ cannot distinguish
between the vertices of $K_5$, see Question
\ref{question-Selections-GRF-v2:1}. Evidently, this can be generalised
for any complete graph $K_{2n+1}$ having an odd number of
vertices.\medskip

For an interesting discussion on round-robin tournaments, the
interested reader is referred to \cite{Harary1966}. 

\subsection*{Communication Networks}

A \emph{communication network} consists of a set of people
$v_1,v_2,\dots, v_n$ such that between some pairs of persons there is
a communication link. Such a link may be either one-way or two-way. A
two-way communication link might be made by telephone or radio, and a
one-way link by sending a messenger, email, lighting a signal light,
setting off an explosion, etc.  The symbol ``$\to$'' is used to
indicate the latter sort of connection; thus $v_i \to v_j$ means that
individual $v_i$ can communicate with individual $v_j$ (in that
direction). The only requirement that is placed on the symbol
``$\to$'' is that $v_i\to v_i$ is false for any $i$; that is, an
individual cannot (or need not) communicate with himself. This brings
an interpretation of a communication network as a directed graph in
which lines $v_iv_j$ and $v_jv_i$ are allowed. 

\subsection*{Communication Matrices}

Communication networks can be represented by means of square matrices
$C$ having only $0$ and $1$ entries, which are called
\emph{communication matrices}. The entry in the $i$-th row and $j$-th
column of $C$ is equal to $1$ if $v_i$ can communicate with $v_j$ (in
that direction) and otherwise equal to $0$. The diagonal entries of
$C$ are all equal to $0$. This is true in general for any
communication matrix, since the matrix restatement of the condition of
a communication network is that all entries $c_{ii}$ of this matrix
$C$ must be $0$. Moreover, it is not hard to see that any matrix
having only $0$ and $1$ entries, and with all zeros down the main
diagonal, is a communication matrix of some network.  Regarding
examples of such communication matrices, the interested reader is
referred to \cite{MR0084454,MR0360048}.

\subsection*{Dominance Relations}

A \emph{dominance relation} is a special kind of a communication
network which is a tournament. Namely, for each $i\neq j$, we have
either $v_i\to v_j$ or $v_j\to v_i$, but not both. Dominance relations
can also be defined by communication matrices, which are called
\emph{dominance matrices}. 
In these terms, the following property of communication networks was
stated in \cite{MR0360048}.

\begin{theorem}[\cite{MR0360048}]
  \label{theorem-Selections-GRF-v3:4}
  Let a communication network of $n$ individuals be such that, for
  every pair of individuals, at least one can communicate in one stage
  with the other. Then there is at least one who can communicate with
  every other person in either one or two stages. Similarly, there is
  at least one person who can be communicated within one or two stages
  by every other person.
\end{theorem}

Every dominance relation satisfies the hypotheses of the above
theorem, but there are communication networks, not dominance relations
(tournaments), that also satisfy these hypotheses. On the other hand,
each such communication network contains a tournament.

\subsection*{Walks, Paths, Cycles and Distance}

A \emph{$\sigma$-walk} from $u\in V$ to $v\in V$ of a tournament
$\sigma$ on $V$ is an alternating sequence of points and lines of the
form $u=u_1, u_1u_2, u_2, \dots, u_{k-1}, u_{k-1}u_k, u_k=v$, or in
terms of the pecking order $<_\sigma$:
\[
  v=u_k\leq_\sigma u_{k-1}\leq_\sigma\dots \leq_\sigma u_2\leq_\sigma u_1=u.
\]
For brevity, a $\sigma$-walk is often written as $u_1u_2\cdots u_k$
since then the lines are clear from context. A $\sigma$-walk in which
all points (and hence all lines) are distinct is called a
\emph{$\sigma$-path}. A $\sigma$-walk of positive length in which only
the first and last points are the same is called a
\emph{$\sigma$-cycle}. The length of a $\sigma$-path or a
$\sigma$-cycle is the number of lines in it. A \emph{complete}
$\sigma$-path or $\sigma$-cycle is the one which contains all the
points of the given digraph. If there is a $\sigma$-path from $u$ to
$v$, then $v$ is called \emph{$\sigma$-reachable} from $u$. The
\emph{$\sigma$-distance} from $u$ to $v$, denoted $d_\sigma (u, v)$,
is the length of a shortest such path. Let us explicitly remark that
$d_\sigma(u,v)$ is not symmetric. For instance, for a set $V$ of three
points $u,v$ and $w$, we may consider the tournament $\sigma$ given by
the cycle $u<_\sigma v<_\sigma w<_\sigma u$ or, in other words, by
$u\to_\sigma w\to_\sigma v\to_\sigma u$, see \eqref{eq:ext-sel-v43:1}.
Then $d_\sigma(u,w)=1$ but $d_\sigma(w,u)=2$.\medskip

We now have the following refinement of Theorem
\ref{theorem-Selections-GRF-v3:1}, the proof can be found in
\cite{MR0084454}, see also the third edition of the book
\cite{MR0360048}. 

\begin{theorem}
  \label{theorem-Selections-GRF-v3:3}
  Let $\sigma$ be a tournament on a finite set $V$. Then the
  $\sigma$-distance from a point with maximal $\sigma$-score to any
  other point is 1 or 2. In particular, each element of $V$ with a
  maximal $\sigma$-score is a $\sigma$-king.
\end{theorem}

The next theorem due to R\'edei \cite{zbMATH03013308} is certainly the
best known result concerning tournaments. Actually, R\'edei showed
that every tournament has an odd number of complete paths.

\begin{theorem}
  \label{theorem-Selections-GRF-v4:2}
  Each tournament has a complete path.
\end{theorem}

Since every tournament has a complete path, it is possible to order
all players in a round-robin competition so that each defeats the
succeeding one.  However, there are two serious difficulties in such a
procedure. First, there is no necessary relation, in general, between
such a ranking of players and their scores. Second, a tournament may
have more than one complete path, so that several different rankings
may be possible. In contrast, if a tournament $\sigma:[V]^2\to V$ is
generating a linear order $<_\sigma$ on $V$, then each point of $V$
can be assigned a distinct rank. Such a tournament is often called
\emph{transitive} because the corresponding selection relation
$\leq_\sigma$ is transitive. 
The following theorem gives a complete description of transitive
tournaments.

\begin{theorem}
  \label{theorem-Selections-GRF-v4:3}
  The following statements are equivalent for any tournament $\sigma$
  on a set $V$ with $n$ points.
  \begin{enumerate}
  \item $<_\sigma$ is transitive.
  \item $<_\sigma$ is acyclic.
  \item $\sigma$ has a unique complete path.
  \item The score sequence of $\sigma$ is $(0, 1, 2,\dots ,n-1)$.
  \item $\sigma$ has $\frac{n(n - 1) (n - 2)}6$ transitive triples.
  \end{enumerate}
\end{theorem}

The next theorem shows that the number of transitive triples in any
tournament may be easily calculated from the score sequence of the
tournament, see Kendall and Babington Smith \cite{MR2761}.

\begin{theorem}
  \label{theorem-Selections-GRF-v4:4}
  The number of transitive triples in a tournament $\sigma$ on $V$
  with a score sequence\,\ $\mathbf{S}_\sigma=(s_1, s_2,\dots , s_n)$\,\
  is\,\ $\sum_{i=1}^n\frac{s_i(s_i-1)}2$.
\end{theorem}

The total number of triples in any tournament with $n$ points is
$\binom{n}3=\frac{n(n-1)(n-2)}6$. Since each triple is either
transitive or cyclic, we obtain the following formula for the number
of cyclic triples.

\begin{corollary}
  \label{corollary-Selections-GRF-v4:3}
  The number of cyclic triples in a tournament $\sigma$ on
  $V$ with a score sequence\,\
  $\mathbf{S}_\sigma=(s_1, s_2,\dots , s_n)$\,\ is\,\
  $\frac{n(n-1)(n-2)}6 - \sum_{i=1}^n \frac{s_i(s_i-1)}2$.
\end{corollary}

The next corollary gives the maximum number of cyclic triples that can
occur in any tournament with a given number of points.

\begin{corollary}
  \label{corollary-Selections-GRF-v4:4}
  Among all the tournaments with $n$ points, the maximum number of
  cyclic triples is\,\ $\frac{n^3-n}{24}$\,\ if $n$ is odd, and\,\
  $\frac{n^3-4n}{24}$\,\ if $n$ is even.
\end{corollary}


\begin{thebibliography}{10}

\bibitem{MR0152989} A.~Arhangel{\cprime}ski{\u\i}, \emph{Ranks of
    systems of sets and dimensionality of spaces},
  Fund. Math. \textbf{52} (1963), 257--275.

\bibitem{MR0296881} A.~V. Arhangel{\cprime}ski{\u\i} and
  V.~V. Filippov, \emph{Spaces with bases of finite rank},
  Mat. Sb. (N.S.) \textbf{87(129)} (1972), 147--158.

\bibitem{artico-marconi:00} G.~Artico and U.~Marconi, \emph{Selections
    and topologically well-ordered spaces}, Topology
  Appl. \textbf{115} (2001), 299--303.

\bibitem{artico-marconi-pelant-rotter-tkachenko:02} G.~Artico,
  U.~Marconi, J.~Pelant, L.~Rotter, and M.~Tkachenko, \emph{Selections
    and suborderability}, Fund. Math. \textbf{175} (2002), 1--33.

\bibitem{bing:51} R.~H. Bing, \emph{Metrization of topological
    spaces}, Canad. J. Math.  \textbf{3} (1951), 175--186.

\bibitem{zbMATH03355968} A.~E. Brouwer, \emph{{On the topological
      characterization of the real line}}, {Math.  Centrum,
    Amsterdam, Afd. zuivere Wisk. ZW 8/71, 6 p. (1971).}

\bibitem{zbMATH03379800}
\bysame, \emph{{A characterization of connected (weakly) orderable
  spaces}}, {Math. Centrum, Amsterdam, Afd. zuivere Wisk. ZW 10/71, 7 p.
  (1971).}

\bibitem{MR3122363} D.~Buhagiar and V.~Gutev, \emph{Selections and
    countable compactness}, Math.  Slovaca \textbf{63} (2013), no.~5,
  1123--1140.

\bibitem{MR3705772} \bysame, \emph{Selections and deleted symmetric
    products}, Tsukuba J. Math.  \textbf{41} (2017), no.~1,
  1--20. 

\bibitem{choban:70a}
M.~Choban, \emph{Many-valued mappings and {B}orel sets. {I}}, Trans. Moscow
  Math. Soc. \textbf{22} (1970), 258--280.

\bibitem{CruzChapital2022} J.~A. Cruz~Chapital, \emph{Continuous
    selections and prime numbers}, Topology Appl. \textbf{305} (2022),
  Id/No 107882, 6 p.

\bibitem{dalen-wattel:73}
{\noopsort{Dalen}}{J. van Dalen} and E.~Wattel, \emph{A topological
  characterization of ordered spaces}, Gen. Top. Appl. \textbf{3} (1973),
  347--354.

\bibitem{douwen:90}
{\noopsort{Douwen}}{E. K. van Douwen}, \emph{Mappings from hyperspaces and
  convergent sequences}, Topology Appl. \textbf{34} (1990), 35--45.

\bibitem{Dow2008}
A.~Dow and S.~Shelah, \emph{Tie-points and fixed-points in {{\(\mathbb
  N^*\)}}}, Topology Appl. \textbf{155} (2008), no.~15, 1661--1671.

\bibitem{MR0235524}
R.~Duda, \emph{On ordered topological spaces}, Fund. Math. \textbf{63} (1968),
  295--309. 

\bibitem{eilenberg:41}
S.~Eilenberg, \emph{Ordered topological spaces}, Amer. J. Math. \textbf{63}
  (1941), 39--45.

\bibitem{MR0261565}
R.~L. Ellis, \emph{Extending continuous functions on zero-dimensional spaces},
  Math. Ann. \textbf{186} (1970), 114--122. 

\bibitem{engelking:89}
R.~Engelking, \emph{General topology, revised and completed edition},
  Heldermann Verlag, Berlin, 1989.

\bibitem{engelking-heath-michael:68}
R.~Engelking, R.~W. Heath, and E.~Michael, \emph{Topological well-ordering and
  continuous selections}, Invent. Math. \textbf{6} (1968), 150--158.

\bibitem{fujii:02}
S.~Fujii, \emph{Characterizations of ordinal spaces via continuous selections},
  Topology Appl. \textbf{122} (2002), 143--150.

\bibitem{fujii-nogura:99}
S.~Fujii and T.~Nogura, \emph{Characterizations of compact ordinal spaces via
  continuous selections}, Topology Appl. \textbf{91} (1999), 65--69.

\bibitem{gutev-nogura-garcia:04a}
S.~Garc\'{\i}a-Ferreira, V.~Gutev, and T.~Nogura, \emph{Extensions of 2-point
  selections}, New Zealand J. Math. \textbf{38} (2008), 1--8.

\bibitem{garcia-ferreira-gutev-nogura-sanchis-tomita:99}
S.~Garc{\'\i}a-Ferreira, V.~Gutev, T.~Nogura, M.~Sanchis, and A.~Tomita,
  \emph{Extreme selections for hyperspaces of topological spaces}, Topology
  Appl. \textbf{122} (2002), 157--181.

\bibitem{garcia-nogura-2011a}
S.~Garc\'{\i}a-Ferreira, K.~Miyazaki, T.~Nogura, and A.~H. Tomita,
  \emph{Topologies generated by weak selection topologies}, Houston J. Math.
  \textbf{39} (2013), no.~4, 1385--1399.

\bibitem{garcia-ferreira-sanchis:04}
S.~Garc{\'\i}a-Ferreira and M.~Sanchis, \emph{Weak selections and
  pseudocompactness}, Proc. Amer. Math. Soc. \textbf{132} (2004), 1823--1825.

\bibitem{MR0063646}
L.~Gillman and M.~Henriksen, \emph{Concerning rings of continuous functions},
  Trans. Amer. Math. Soc. \textbf{77} (1954), 340--362. 

\bibitem{glicksber:59}
I.~Glicksberg, \emph{Stone-{{\v C}}ech compactifications of products}, Trans.
  Amer. Math. Soc. \textbf{90} (1959), 369--382.

\bibitem{groot:56}
J.~de~Groot, \emph{Non-{Archimedean} metrics in in topology}, Proc. Amer. Math.
  Soc. \textbf{9} (1956), 948--953.
  
\bibitem{gutev:00e}
V.~Gutev, \emph{Fell continuous selections and topologically well-orderable
  spaces {II}}, {P}roceedings of the {N}inth {P}rague {T}opological {S}ymposium
  (2001), Topology Atlas, Toronto, 2002, pp.~157--163 (electronic).

\bibitem{gutev:05a}
\bysame, \emph{Approaching points by continuous selections}, J. Math. Soc.
  Japan (2006), no.~4, 1203--1210.

\bibitem{gutev:07a}
\bysame, \emph{Weak orderability of second countable spaces}, Fund. Math.
  \textbf{196} (2007), no.~3, 275--287.

\bibitem{gutev:07b}
\bysame, \emph{Orderability in the presence of local compactness}, J. Math.
  Soc. Japan \textbf{60} (2008), no.~3, 741--766.

\bibitem{gutev:08a}
\bysame, \emph{Selections and hyperspaces of finite sets}, Topology Appl.
  \textbf{157} (2010), no.~1, 83--89.

\bibitem{gutev:11}
\bysame, \emph{Closed graph multi-selections}, Fund. Math. \textbf{211} (2011),
  85--99.

\bibitem{gutev-2013springer}
\bysame, \emph{Selections and hyperspaces}, Recent Progress in General
  Topology, {III} (K.~P. Hart, J.~van Mill, and P.~Simon, eds.), Atlantis
  Press, Springer, 2014, pp.~535--579.

\bibitem{MR3430989}
\bysame, \emph{Selection topologies}, Topology Appl. \textbf{196} (2015),
  Part B, 458--467.

\bibitem{MR3640030}
\bysame, \emph{King spaces and compactness}, Acta Math. Hungar. \textbf{152}
  (2017), no.~1, 1--10. 

\bibitem{Gutev2018b}
\bysame, \emph{Coarser compact topologies}, Matemati{\v{c}}ki Vesnik
\textbf{70} (2018), no.~4, 350--363.

\bibitem{zbMATH06882273}
\bysame, \emph{Usco sections and {C}hoquet-completeness}, Fund. Math.
  \textbf{242} (2018), no.~1, 93--102.

\bibitem{Gutev2021a}
\bysame, \emph{Weak selections and suborderable metrizable spaces}, Topology
  Appl. \textbf{301} (2021), Id/No 107506, 15 p.

\bibitem{gutev-nogura:01a}
V.~Gutev and T.~Nogura, \emph{Selections and order-like relations}, Appl. Gen.
  Topol. \textbf{2} (2001), 205--218.

\bibitem{gutev-nogura:00b}
\bysame, \emph{Vietoris continuous selections and disconnectedness-like
  properties}, Proc. Amer. Math. Soc. \textbf{129} (2001), 2809--2815.

\bibitem{gutev-nogura:00d}
\bysame, \emph{Fell continuous selections and topologically well-orderable
  spaces}, Mathematika \textbf{51} (2004), 163--169.

\bibitem{gutev-nogura:02a}
\bysame, \emph{Some problems on selections for hyperspace topologies}, Appl.
  Gen. Topol. \textbf{5} (2004), no.~1, 71--78.

\bibitem{gutev-nogura:03a}
\bysame, \emph{Selection pointwise-maximal spaces}, Topology Appl.
  \textbf{146-147} (2005), 397--408.

\bibitem{gutev-nogura:03b}
\bysame, \emph{A topology generated by selections}, Topology Appl. \textbf{153}
  (2005), 900--911.

\bibitem{gutev-nogura:06b}
\bysame, \emph{Selection problems for hyperspaces}, Open {P}roblems in
  {T}opology 2 (E.~Pearl, ed.), Elsevier BV., Amsterdam, 2007, pp.~161--170.

\bibitem{gutev-nogura:08a}
\bysame, \emph{Weak selections and flows in networks}, Comment. Math. Univ.
  Carolin. \textbf{49} (2008), no.~3, 509--517.

\bibitem{gutev-nogura:09a}
\bysame, \emph{Weak orderability of topological spaces}, Topology Appl.
  \textbf{157} (2010), 1249--1274.

\bibitem{haar-konig:10}
A.~Haar and D.~K{\"o}nig, \emph{{\"U}ber einfach geordnete {M}engen}, J.
  f{\"u}r die riene und engew. Math \textbf{139} (1910), 16--28.

\bibitem{Harary1966}
F.~Harary and L.~Moser, \emph{The theory of round robin tournaments}, Amer.
  Math. Monthly \textbf{73} (1966), 231--246.

\bibitem{hausdorff:34}
F.~Hausdorff, \emph{{\"U}ber innere {A}bbildungen}, Fund. Math. \textbf{23}
  (1934), 279--291.

\bibitem{herrlich:65}
H.~Herrlich, \emph{Ordnungsf{\"a}higkeit total-diskontinuierlicher
  {R}{\"a}ume}, Math. Ann. \textbf{159} (1965), 77--80.

\bibitem{MR0185564}
\bysame, \emph{Ordnungsf{\"a}higkeit zusammenh{\"a}ngender {R}{\"a}ume},
  Fund. Math. \textbf{57} (1965), 305--311.

\bibitem{hrusak-martinez:09}
M.~Hru{\v s}{\'a}k and I.~Mart{\'{\i}}nez-Ruiz, \emph{Selections and weak
  orderability}, Fund. Math. \textbf{203} (2009), 1--20.

\bibitem{Jiang2007}
N.~Jiang, \emph{Vietoris continuous selections on finite sets}, Questions
  Answers Gen. Topology \textbf{25} (2007), no.~2, 135--142.

\bibitem{keesling:70}
J.~Keesling, \emph{On the equivalence of normality and compactness in
  hyperspaces}, Pacific J. Math. \textbf{33} (1970), 657--667.

\bibitem{MR0084454}
J.~G. Kemeny, J.~L. Snell, and G.~L. Thompson, \emph{Introduction to finite
  mathematics}, Prentice-Hall, Inc., Englewood Cliffs, N. J., 1957.

\bibitem{MR0360048}
\bysame, \emph{Introduction to finite mathematics}, third ed., Prentice-Hall,
  Inc., Englewood Cliffs, N.J., 1974.

\bibitem{MR2761}
M.~G. Kendall and B.~B. Smith, \emph{On the method of paired comparisons},
  Biometrika \textbf{31} (1940), 324--345.

\bibitem{zbMATH03030217}
C.~{Kuratowski}, \emph{{Sur la compactification des espaces \`a connexite
  n-dimensionnelle}}, {Fund. Math.} \textbf{30} (1938), 242--246.

\bibitem{MR0259835}
K.~Kuratowski, \emph{Topology. {V}ol. {II}}, New edition, revised and
  augmented. Translated from the French by A. Kirkor, Academic Press, New
  York-London; Pa{\'n}stwowe Wydawnictwo Naukowe Polish Scientific Publishers,
  Warsaw, 1968.

\bibitem{zbMATH03023961}
{\DJ}.~{Kurepa}, \emph{{Le probl\`eme de Souslin et les espaces abstraits}},
  {C. R. Acad. Sci., Paris} \textbf{203} (1936), 1049--1052.

\bibitem{zbMATH03027256}
\bysame, \emph{{Sur les classes $({\mathcal E})$ et $({\mathcal D})$}}, {Publ.
  Math. Univ. Belgrade} \textbf{5} (1936), 124--132.

\bibitem{MR0102789}
\bysame, \emph{Sur l'{\'e}cart abstrait}, Glasnik Mat.-Fiz. Astr. Ser. II.
  \textbf{11} (1956), 105--134.

\bibitem{MR1580869}
J.~K{\"u}rsch{\'a}k, \emph{{\"U}ber {L}imesbildung und allgemeine
  {K}{\"o}rpertheorie}, J. Reine Angew. Math. \textbf{142} (1913), 211--253.

\bibitem{MR0041412}
H.~G. Landau, \emph{On dominance relations and the structure of animal
  societies. {I}. {E}ffect of inherent characteristics}, Bulletin of
  Mathematical Biophysics \textbf{13} (1951), 1--19.

\bibitem{MR54932}
\bysame, \emph{On dominance relations and the structure of animal societies.
  {II}. {S}ome effects of possible social factors}, Bull. Math. Biophys.
  \textbf{13} (1951), 245--262.

\bibitem{MR54933}
\bysame, \emph{On dominance relations and the structure of animal societies.
  {III}. {T}he condition for a score structure}, Bull. Math. Biophys.
  \textbf{15} (1953), 143--148.

\bibitem{MR0138089}
I.~L. Lynn, \emph{Linearly orderable spaces}, Proc. Amer. Math. Soc.
  \textbf{13} (1962), 454--456.

\bibitem{MR0093753}
M.~J. Mansfield, \emph{Some generalizations of full normality}, Trans. Amer.
  Math. Soc. \textbf{86} (1957), 489--505.

\bibitem{MR567954}
S.~B. Maurer, \emph{The king chicken theorems}, Math. Mag. \textbf{53} (1980),
  no.~2, 67--80. 

\bibitem{michael:51}
E.~Michael, \emph{Topologies on spaces of subsets}, Trans. Amer. Math. Soc.
  \textbf{71} (1951), 152--182.

\bibitem{MR0077107}
\bysame, \emph{Continuous selections. {I}}, Ann. of Math. (2) \textbf{63}
  (1956), 361--382.

\bibitem{michael:56}
\bysame, \emph{Selected selections theorems}, Amer. Math. Monthly \textbf{63}
  (1956), 233--238.

\bibitem{mill-pelant-pol:96}
J.~v. Mill, J.~Pelant, and R.~Pol, \emph{Selections that characterize
  topological completeness}, Fund. Math. \textbf{149} (1996), 127--141.

\bibitem{mill-wattel:81}
{\noopsort{Mill}}{J. van Mill} and E.~Wattel, \emph{Selections and
  orderability}, Proc. Amer. Math. Soc. \textbf{83} (1981), no.~3, 601--605.

\bibitem{mill-wattel:84}
\bysame, \emph{Orderability from selections: Another solution to the
  orderability problem}, Fund. Math. \textbf{121} (1984), 219--229.

\bibitem{miyazaki:01b}
K.~Miyazaki, \emph{Continuous selections on almost compact spaces}, Sci. Math.
  Jpn. \textbf{53} (2001), 489--494.

\bibitem{MR0035982}
A.~F. Monna, \emph{Remarques sur les m{\'e}triques non-archi-m{\'e}diennes.
  {I}}, Nederl. Akad. Wetensch., Proc. \textbf{53} (1950), 470--481 =
  Indagationes Math. 12, 122--133 (1950). 

\bibitem{MR0035983}
\bysame, \emph{Remarques sur les m{\'e}triques non-archi-m{\'e}diennes. {II}},
  Nederl. Akad. Wetensch., Proc. \textbf{53} (1950), 625--637 = Indagationes
  Math. 179--191 (1950).

\bibitem{Motooka2019}
K.~Motooka, \emph{Weak selections and countable compactness}, Topology Proc.
  \textbf{53} (2019), 123--129.

\bibitem{MR2944781}
M.~Nagao and D.~Shakhmatov, \emph{On the existence of kings in continuous
  tournaments}, Topology Appl. \textbf{159} (2012), no.~13, 3089--3096.

\bibitem{nagata:62}
J.~Nagata, \emph{On dimension and metrization}, in: General {T}opology and its
  {R}elations to {M}odern {A}nalysis and {A}lgebra, Academic Press, 1962.

\bibitem{nogura-shakhmatov:97a}
T.~Nogura and D.~Shakhmatov, \emph{Characterizations of intervals via
  continuous selections}, Rendiconti del Circolo Matematico di Palermo, Serie
  II, \textbf{46} (1997), 317--328.

\bibitem{MR1555153}
A.~Ostrowski, \emph{{\"U}ber einige {L}{\"o}sungen der {F}unktionalgleichung
  {$\psi(x)\cdot\psi(x)=\psi(xy)$}}, Acta Math. \textbf{41} (1916), no.~1,
  271--284.

\bibitem{oxtoby:60}
J.~C. Oxtoby, \emph{Cartesian products of {Baire} spaces}, Fund. Math.
  \textbf{49} (1960), 157--166.

\bibitem{MR0054940}
P.~Papi\v{c}, \emph{Sur les espaces admettant une base ramifi{\'e}e de
  voisinages}, Hrvatsko Prirod. Dru\v stvo. Glasnik Mat.-Fiz. Astr. Ser.
  II. \textbf{8} (1953), 30--43.

\bibitem{MR0073964}
\bysame, \emph{Sur une classe d'espaces abstraits}, Hrvatsko Prirod. Dru\v
  stvo. Glasnik Mat.-Fiz. Astr. Ser. II. \textbf{9} (1954), 197--216.

\bibitem{przymusinski:78a}
T.~Przymusi{\'n}ski, \emph{Collectionwise normality and absolute retracts},
  Fund. Math. \textbf{98} (1978), 61--73.

\bibitem{purisch:77}
S.~Purisch, \emph{The orderability and suborderability of metrizable spaces},
  Trans. Amer. Math. Soc. \textbf{226} (1977), 59--76.

\bibitem{zbMATH03013308}
L.~{R\'edei}, \emph{{Ein kombinatorischer Satz}}, {Acta Litt. Sci. Szeged}
  \textbf{7} (1934), 39--43.

\bibitem{Sapirovskii1975}
B.~E. {\v{S}apirovski\u{\i}}, \emph{The imbedding of extremally disconnected
  spaces in bicompacta. {$b$}-points and weight of pointwise normal spaces},
  Dokl. Akad. Nauk SSSR \textbf{223} (1975), no.~5, 1083--1086.

\bibitem{silverman1962problem}
D.~Silverman, \emph{Problem 463}, Math. Mag \textbf{35} (1962), 189.

\bibitem{MR0257985}
L.~A. Steen, \emph{A direct proof that a linearly ordered space is hereditarily
  collectionwise normal}, Proc. Amer. Math. Soc. \textbf{24} (1970), 727--728.

\bibitem{velichko:75}
N.~Velichko, \emph{On spaces of closed subsets}, Siberian Math. J. \textbf{16}
  (1975), 627--629, (in Russian).

\bibitem{venkataraman-rajagopalan-soundararajan:72}
M.~Venkataraman, M.~Rajagopalan, and T.~Soundararajan, \emph{Orderable
  topological spaces}, Gen. Top. Appl. \textbf{2} (1972), 1--10.
\end{thebibliography}

\newcommand{\noopsort}[1]{} \newcommand{\singleletter}[1]{#1}
\providecommand{\bysame}{\leavevmode\hbox to3em{\hrulefill}\thinspace}
\smallskip

\end{document}